\documentclass[11pt,a4paper]{amsart}
\usepackage{amsmath,amsthm,thmtools, amsfonts,amscd,amssymb,enumitem,latexsym,mathrsfs,
mathtools,stmaryrd,tikz-cd,bm,verbatim}

\usepackage{hyperref}
\usepackage{cleveref}

\newtheorem{theorem}{Theorem}[section]
\newtheorem{corollary}[theorem]{Corollary}
\newtheorem{lemma}[theorem]{Lemma}
\newtheorem{proposition}[theorem]{Proposition}

\theoremstyle{definition}
\newtheorem{definition}[theorem]{Definition}
\newtheorem{remark}[theorem]{Remark}

\newtheorem*{definition*}{Definition}

\newcounter{theoremintro}

\numberwithin{equation}{section}

\allowdisplaybreaks

\newcommand{\mc}[1]{\mathcal{#1}}
\newcommand{\bb}[1]{\mathbb{#1}}

\def\acts{\curvearrowright}

\allowdisplaybreaks

\title{Paper-folding models for the CAR algebra}

\author{Grigoris Kopsacheilis}
\address{Grigoris Kopsacheilis, KU Leuven Department of Mathematics, Celestijnenlaan 200B box 2400, BE-3001 Leuven, Belgium.}
\email{gkopsach@kuleuven.be}
\urladdr{https://sites.google.com/view/gkopsach}

\author{Wilhelm Winter}
\address{Wilhelm Winter, Mathematisches Institut, Fachbereich Mathematik und Informatik der
Universit\"at M\"unster, Einsteinstrasse 62, 48149 M\"unster, Germany.}
\email{wwinter@uni-muenster.de}
\urladdr{https://wilhelm-winter.de/}
\begin{document}
\begin{abstract}
We show that the CAR algebra admits a Cantor spectrum $\mathrm{C}^\ast$-diagonal that is not conjugate to the standard AF diagonal. We obtain this by classification theory of $\mathrm{C}^\ast$-algebras, and the diagonal arises by realising the CAR algebra as the crossed product of a free minimal action on the Cantor space, where the acting group is the product of a locally finite group with the infinite dihedral group. The main ingredient in the construction is a binary subshift associated to the well-known regular paper-folding sequence.

Moreover, we show that the CAR algebra in fact admits countably many, pairwise non-conjugate, Cantor spectrum diagonals which are distinguished by the different values of their diagonal dimension, as defined by Li, Liao and the second named author.

\end{abstract}
\maketitle

\section{Introduction}

\renewcommand{\thetheorem}{\Alph{theorem}}
\setcounter{theorem}{0}

\noindent The archetypical and quite inexhaustible source of examples of operator algebras are dynamical structures, interpreted broadly,\footnote{This could be classical topological or measurable dynamical systems, say with countable discrete acting groups, or also much more general setups, as long as they can be encoded as topological or measurable groupoids.} and the latter often provide highly efficient tools to study the former. In the case of classical dynamical systems with amenable acting groups, the associated operator algebras are particularly well-behaved and fall within the scope of spectacular classification results, most notably Connes’ classification of injective $\mathrm{II}_1$ factors on the von Neumann algebra side, and, much more recently, the classification of simple, nuclear and $\mc{Z}$-stable $\mathrm{C}^\ast$-algebras.

In the measurable, or von Neumann algebra case, such classification results tend to be particularly clean: there is only one injective $\mathrm{II}_1$ factor by Connes’ work, and by the Connes--Feldman--Weiss theorem every countable amenable equivalence relation is generated by a single transformation. It then follows that the hyperfinite $\mathrm{II}_1$ factor up to conjugacy has only one Cartan subalgebra (as introduced by Vershik and Feldman--Moore in \cite{Ver71, FelMoo77i-ii}) which is a very strong uniqueness result for underlying dynamical structures in this setting.

In the topological, or $\mathrm{C}^\ast$-algebra situation, one has to deal with a more complicated picture, keeping track of both dynamical and topological information in terms of $\mathrm{K}$-theoretic invariants. There is also the much more fundamental problem that we do not know whether every simple, separable, nuclear, and $\mc{Z}$-stable\footnote{$\mc{Z}$-stability is tensorial absorption of the Jiang--Su algebra $\mc{Z}$.} $\mathrm{C}^\ast$-algebra even has an underlying dynamical structure in the form of a Cartan subalgebra (this time in the $\mathrm{C}^\ast$-algebraic sense of Kumjian and Renault; see \cite{Kum86, Ren08}). This question is in fact equivalent to the all-intriguing UCT problem, which asks whether every separable nuclear $\mathrm{C}^\ast$-algebra is $\mathrm{KK}$-equivalent (a weak form of homotopy equivalence) to an abelian $\mathrm{C}^\ast$-algebra. While the existence of \emph{some} Cartan subalgebra is necessary for the classifiability of the ambient $\mathrm{C}^\ast$-algebra, even when the latter is quite well-understood as a simple nuclear $\mathrm{C}^\ast$-algebra, we know very little about the entirety of its underlying dynamical structures. The present paper makes progress on this problem for the CAR algebra, the UHF algebra $\mathrm{M}_{2^\infty} = \mathrm{M}_2\otimes\mathrm{M}_2\otimes\dots$ as introduced by Glimm, which, in a sense, can be thought of as the most elementary $\mathrm{C}^\ast$-algebra analogue of the hyperfinite $\mathrm{II}_1$ factor. Let us be more precise now.

A \emph{Cartan sub-$\mathrm{C}^\ast$-algebra} is a $\mathrm{C}^\ast$-pair $(D\subset A)$ where $A$ is generated by the set of \emph{normalisers} of $D$ in $A$ (i.e.\ those $v\in A$ such that $vDv^*\cup v^*Dv \subset D$), $D$ is maximal abelian in $A$ (a \emph{masa}) and there is a unique faithful conditional expectation $A\to D$. When $(D \subset A)$ moreover has the \emph{unique extension property}, in the sense that pure states on $D$ extend uniquely to states on $A$, we say that it is a \emph{$\mathrm{C}^\ast$-diagonal}. This additional feature ensures precisely that the underlying dynamics have trivial isotropy, i.e.\ they are free, a trait that justifies why $\mathrm{C}^\ast$-diagonals are in general considered more robust than Cartan pairs.

Identifying Cartan sub-$\mathrm{C}^\ast$-algebras and diagonals in a $\mathrm{C}^\ast$-algebra is a topic that has recently been intensively studied, motivated by the various implications of the presence of such intrinsic dynamics. A natural companion question in this framework is determining whether two Cartan sub-$\mathrm{C}^\ast$-algebras with homeomorphic spectra are \emph{conjugate}, in the sense that there is an automorphism of the ambient $\mathrm{C}^\ast$-algebra carrying one masa to the other. Establishing non-conjugacy is a problem that often proves to be rather difficult, and a common strategy for this is to distinguish the two pairs by topological properties of the spaces of state-extensions of the evaluations over points in the spectra of the masas -- a method that is of course completely inapplicable when dealing with $\mathrm{C}^\ast$-diagonals, the case of which requires the development of more sophisticated means of distinction.

For AF algebras, there are standard diagonals with totally disconnected spectra (called \emph{AF diagonals}) obtained as inductive limits of masas in finite-dimensional $\mathrm{C}^\ast$-algebras with connecting maps that preserve normalisers -- and these are all conjugate to each other (see \cite[Theorem~5.7]{Pow92}, where AF diagonals were called \emph{canonical regular masas}). While not at all obvious, there is, in general, a rich supply of other, more exotic diagonals: in his celebrated paper \cite{Bla90}, Blackadar showed that the CAR algebra surprisingly admits uncountably many pairwise non-conjugate $\mathrm{C}^\ast$-diagonals, all with the same $1$-dimensional spectrum (the product of the circle $S^1$ with the Cantor space), as well as that there exist diagonals with spectra that have both singletons and intervals as connected components. A problem that was left unresolved however is whether all $\mathrm{C}^\ast$-diagonals of the CAR algebra \emph{with Cantor spectrum} need to be (conjugate to) an AF diagonal. This question has come to the forefront in view of increasing efforts to develop a structure theory for Cartan and diagonal subalgebras of a given classifiable $\mathrm{C}^\ast$-algebra; cf.\ Problems XLVII and XLVIII of \cite{SchTikWhi25} and the remarks in between.

The particular interest in the Cantor space in this context comes from the fact that it is both universal and well accessible to $\mathrm{K}$-theoretic methods; it has the additional feature that it is invariant under taking tensor products. Our results confirm the existence of an abundance of such diagonals.

\begin{theorem}\label{intro:thm1}
The CAR algebra $\mathrm{M}_{2^\infty}$ admits a Cantor spectrum $\mathrm{C}^\ast$-diagonal that is not an AF diagonal.
\end{theorem}

While Blackadar's method employed Elliott's classification of AF algebras, our approach uses the more recent advances in the classification theory of stably finite $\mathrm{C}^\ast$-algebras; cf.\ \cite{Whi23, Win18, TikWhiWin17}. More precisely, we first construct the crossed product $\mathrm{C}^\ast$-algebra of a certain action of the infinite dihedral group $\bb{Z}\rtimes\bb{Z}_2$ on the Cantor space and we use the results of \cite{Tho10} (largely based on \cite{Nat85, BraEvaKis93}) to compute the $\mathrm{K}$-theory. Since we are aiming for the construction of a $\mathrm{C}^\ast$-\emph{diagonal} in a simple $\mathrm{C}^\ast$-algebra it is imperative that we start with a \emph{free} minimal action, and so the natural initial question is, which Cantor minimal systems ($\bb{Z}$-actions at this point) extend to \emph{free} actions of the infinite dihedral group. While the existence of such actions is well-understood (e.g.\ by Thomsen's analysis of $1$-dimensional solenoids \cite{Tho10}), producing an explicit example is non-trivial, let alone pinpointing one with the desired $\mathrm{K}$-theoretic features. In fact, a compactness argument shows that equicontinuous $\bb{Z}$-actions can only extend to non-free actions of $\bb{Z}\rtimes\bb{Z}_2$, so odometer actions are notably of no use for this problem  -- observe however that the dyadic odometer together with the involution on $\{\mathsf{0},\mathsf{1}\}^\bb{N}$ that swaps zeroes and ones defines an action of $\bb{Z}\rtimes\bb{Z}_2$ on the Cantor space that does give rise to a Cartan sub-$\mathrm{C}^\ast$-algebra which is not a diagonal in a simple monotracial AF algebra with a free abelian summand in its $\mathrm{K}_0$-group (and which is thus not UHF) \cite{BraEvaKis93}.

Despite this obscurity, it turns out that such examples can be obtained from expansive actions -- in particular, we consider a minimal subshift over the binary alphabet $\{\mathsf{0,1}\}$ the language of which is the set of finite subwords appearing in the classic \emph{paper-folding sequence} that curiously makes its appearance in this construction (cf.\ \Cref{sec:pf}): this is the self-similar sequence of ones and zeroes obtained by repeatedly folding a strip of paper each time by a right turn, then unfolding it and marking $\mathsf{1}$ or $\mathsf{0}$ for every right or left (respectively) turn that occurs in the resulting crest -- see \cite{DekFravdP82, Tab14} for some intriguing characteristics of the paper-folding sequence. The involution that implements the action of the $\bb{Z}_2$-subgroup in $\bb{Z}\rtimes\bb{Z}_2$ is given by \emph{anti-reversing}, i.e.\ flipping bi-infinite sequences with respect to the $0$th coordinate and swapping all $\mathsf{0}$s and $\mathsf{1}$s. Freeness of the resulting action is guaranteed by the fact that the paper-folding sequence contains only finitely many \emph{anti-palindromes} (i.e.\ finite subwords fixed under anti-reversal), a fact that, albeit not too surprising, is far from obvious (cf.\ \cite{BerBoaCarFag09}, \Cref{prop:pf}).

Utilising the techniques of \cite{DurHosSka99} for the computation of the $\mathrm{K}$-theory of \emph{substitutional systems}, we devise an ad-hoc method to compute the $\mathrm{K}$-theory of the paper-folding subshift, and we then calculate the $\mathrm{K}$-theory of the crossed product of the obtained free minimal action of $\bb{Z}\rtimes\bb{Z}_2$ via \cite{Tho10} (cf.\ \cite{Nat85, Kum90, BraEvaKis93}), in passing answering a question raised by Scarparo in \cite[Remark~2.3~(ii)]{Sca23}. This crossed product is a classifiable\footnote{By \emph{classifiable} we mean unital, separable, simple, nuclear, satisfying the universal coefficient theorem (UCT), and $\mc{Z}$-stable. Equivalently, $\mc{Z}$-stability can herein be replaced with having finite nuclear dimension; cf.\ \cite{Win12, CasEviTikWhiWin21}. Also see \cite{TomWin13} for classifiability of integer crossed product $\mathrm{C}^\ast$-algebras. \label{footnote}} $\mathrm{C}^\ast$-algebra, and its $\mathrm{K}$-theory turns out to be $\bb{Z}_2\oplus\bb{Z}[\tfrac{1}{2}]$ with order structure determined by the dyadic rationals, and so the torsion part indicates that we have not yet acquired an AF algebra -- nevertheless, upon tensoring with $\mathrm{M}_{2^\infty}$, the K{\"u}nneth formula ensures that the $2$-torsion summand is annihilated by $2$-divisibility of $\mathrm{K}_0(\mathrm{M}_{2^\infty})$. At this point classification theory of (stably finite) nuclear $\mathrm{C}^\ast$-algebras (cf.\ \cite{Lin01, RorSto02, Whi23, Win18}) allows us to conclude that we have constructed a Cantor spectrum diagonal in $\mathrm{M}_{2^\infty}$, but it yet remains to argue that this is not conjugate to an AF diagonal. To this end we use a result of Archbold and Kumjian \cite{ArcKum86} which says that intermediate sub-$\mathrm{C}^\ast$-algebras of AF diagonals are themselves AF algebras.

Employing the theory of \emph{diagonal dimension} of $\mathrm{C}^\ast$-pairs that was developed by Li, Liao and the second named author in \cite{LiLiaWin23}, by considering tensor powers of the non-AF diagonal, and since $\mathrm{M}_{2^\infty}$ is (strongly) self-absorbing, we obtain the following strengthening of \Cref{intro:thm1}, which in passing resolves a problem raised in \cite[Remark~6.10]{LiLiaWin23}.

\begin{theorem}\label{intro:thm2}
For each $n\in\{0,1,2,\dots,\infty\}$, there is a Cantor spectrum diagonal $(D \subset \mathrm{M}_{2^\infty})$ such that
\[
\dim_{\mathrm{diag}}(D \subset \mathrm{M}_{2^\infty})=n.
\]

In particular, the CAR algebra admits (at least) countably many pairwise non-conjugate $\mathrm{C}^\ast$-diagonals each of which has Cantor spectrum.
\end{theorem}

The existence of infinitely many pairwise non-conjugate $\mathrm{C}^\ast$-diagonals with Cantor spectrum in a $\mathrm{C}^\ast$-algebra also occurs in other contexts -- for example, the Bunce--Deddens algebra of type $2^\infty$ (i.e.\ the crossed product of the dyadic odometer) admits uncountably many pairwise non-conjugate Cantor spectrum diagonals. Indeed, Boyle and Handelman have shown in \cite{BoyHan94} that for any $0\le t\le \infty$ there is a Cantor minimal system that is strongly orbit equivalent to the dyadic odometer and with topological entropy $t$, so each of these systems gives rise to a Cantor spectrum diagonal therein (cf.\ \cite{GioPutSka95}), and pairwise non-conjugacy of these is witnessed by topological entropy thanks to a result of Boyle--Tomiyama \cite[Theorem~10.4]{Put18}; in fact, the analogous statement is true in any crossed product $\mathrm{C}^\ast$-algebra of a Cantor minimal system, due to a result of Sugisaki \cite{Sugisaki07}. What is particularly remarkable in \Cref{intro:thm2} is that this phenomenon occurs in an AF (in fact a UHF) algebra, which can never be realised as a crossed product of an integer action due to a $\mathrm{K}_1$-obstruction.

Constructing non-standard Cartan sub-$\mathrm{C}^\ast$-algebras in AF algebras is a topic that has already been explored to some extent, employing methods inspired by graph $\mathrm{C}^\ast$-algebras which are quite different from ours. Classical Cuntz--Krieger algebras of acyclic graphs do not quite do the trick, since they either yield AF diagonals or purely infinite (hence very much non-AF) $\mathrm{C}^\ast$-algebras with non-diagonal Cartans. The latter happens due to non-trivial isotropy of the underlying groupoids in that case, which is typical behaviour for graph $\mathrm{C}^\ast$-algebras. However, the generalisation of Kumjian and Pask \cite{KumPas00} to \emph{higher rank graphs} gives ground for producing more complex Cartan pairs, although with evidently greater difficulty towards deducing that the ambient $\mathrm{C}^\ast$-algebras are AF (let alone UHF). This topic was thoroughly studied by Evans and Sims in \cite{EvaSim12}, where in particular the authors give an application of their methods to produce an example of a Cartan sub-$\mathrm{C}^\ast$-algebra (with totally disconnected spectrum) that is not a diagonal in a $\mathrm{C}^\ast$-algebra which seems to resemble the CAR algebra in many ways (cf.\ \cite[Proposition~6.12]{EvaSim12}). Whether this ambient $\mathrm{C}^\ast$-algebra in their example is indeed $\mathrm{M}_{2^\infty}$ was left unanswered; however, one can bypass this obstacle by simply tensoring with $\mathrm{M}_{2^\infty}$. We address this in \Cref{rmk:graphs}. A similar construction was carried out by Mitscher and Spielberg in a slightly more general setting in \cite{MitSpi22}, where the authors obtained Cantor spectrum Cartan sub-$\mathrm{C}^\ast$-algebras that are not diagonals, in  simple monotracial AF algebras that are Morita equivalent to the Effros--Shen continued fraction algebras \cite{EffShe80}.

On the side of purely infinite $\mathrm{C}^\ast$-algebras, Cantor spectrum diagonals have been produced in the Cuntz algebra $\mc{O}_2$ by Sibbel and the second named author \cite{SibWin24}, and \cite{EviSib25} generalises this construction to $\mc{O}_n$ for $n\ge2$ (leaving the case $n=\infty$ open). The present work also complements these results in the stably finite case, which altogether indicate that, at least among the known strongly self-absorbing $\mathrm{C}^\ast$-algebras, there is a vast supply of diagonals with universal spectra (see also \cite{DeePutStr18} for diagonals with higher dimensional spectra in the Jiang--Su algebra) -- an observation that may also contribute to our understanding of the notorious UCT problem for nuclear $\mathrm{C}^\ast$-algebras.

\subsection*{Acknowledgements}
This work was funded by the Deutsche Forschungsgemeinschaft
(DFG, German Research Foundation) under Germany's Excellence Strategy EXC 2044-390685587,
Mathematics M{\"u}nster: Dynamics--Geometry--Structure, by the SFB 1442 of the DFG, and by ERC Advanced Grant 834267 -- AMAREC.

\section{The paper-folding subshift}\label{sec:pf}

\renewcommand{\thetheorem}{\thesection.\arabic{theorem}}
\setcounter{theorem}{0}

\noindent Given a finite alphabet $\sf{A}$, we consider the set $\sf{A}^\bb{Z}$ of bi-infinite sequences in $\sf{A}$. Endowed with the product topology (where $\sf{A}$ is considered as a discrete space) $\sf{A}^\bb{Z}$ is a Cantor space, and a compatible metric for its topology is given by
\begin{equation}\label{eq:metric}
d(x,y) = 2^{-n},\quad x\ne y\in \sf{A}^\bb{Z},
\end{equation}
where $n\coloneqq\min\{ k\ge0: x_k\ne y_k \text{ or }x_{-k}\ne y_{-k}\}$ with $x=(x_i)_{i\in\bb{Z}}$ and $y=(y_i)_{i\in\bb{Z}}$. Similarly, let $\sf{A}^\bb{N}$ be the set of infinite sequences in $\sf{A}$ indexed by $\bb{N}=\{0,1,2,3,\dots\}$, which is also a Cantor space with the product topology.

We also consider the free monoid generated by $\sf{A}$, which is the set $\sf{A}^*$ of all finite words with letters from $\sf{A}$, endowed with the (associative) operation of concatenation, and with neutral element the empty word. For a word $v\in\sf{A}^*$ we let $|v|$ denote the length of $v$, and for $a\in \mathsf{A}$ we denote by $|v|_a$ the number of occurrences of $a$ in $v$. A word of length $0$ is by convention the empty word. For $w=x_1\dots x_{n}\in\sf{A}^*$ and $1\le k \le \ell \le n$, we write $w_{[k,\ell]}$ for the segment $x_kx_{k+1}\dots x_\ell$ of $w$, and we extend this notation in the obvious way to infinite or bi-infinite sequences on $\sf{A}$.

For $v\in \sf{A}^*$ and $w$ a word from $\sf{A}$ that is either finite, infinite, or bi-infinite, we write $v\prec w$ when $v$ is a \emph{subword} of $w$, i.e.\ $v$ appears as a segment in $w$. Clearly $\prec$ is a transitive relation on $\sf{A}^*$. An infinite or bi-infinite word $w$ is called \emph{uniformly recurrent} when for any $v\prec w$ there is $\ell\in\mathbb{N}$ such that if $v'\prec w$ with $|v'|\ge\ell$, then $v\prec v'$. In other words, $w$ is uniformly recurrent precisely when every finite subword of $w$ appears infinitely many times in $w$ with the gaps between any consecutive occurrences having bounded length. We say that $w\in\mathsf{A}^\bb{N}$ is \emph{periodic} when there is a finite word that is being repeated from some slot onward, i.e.\ $w=v_0vvv\dots$, for some finite words $v_0$ and $v$.

For $X \subset \sf{A}^\bb{Z}$ or $X \subset \sf{A}^\bb{N}$ or $X\subset \sf{A}^*$, we write $\mc{L}(X)$ for the \emph{language} of $X$, namely
\begin{equation}
\mc{L}(X)\coloneqq \{ v\in\mathsf{A}^*: \text{there is }x\in X\text{ such that }v\prec x\}.
\end{equation}
We write 
\begin{equation}\label{eq:shiftmap}
\varphi_\mathsf{A}\colon\sf{A}^\bb{Z}\to\sf{A}^\bb{Z}
\end{equation}
for the shift map $\varphi_\mathsf{A}(x_j)_{j\in\bb{Z}} \coloneqq (x_{j+1})_{j\in\bb{Z}}$ which we simply denote by $\varphi$ when there is no confusion. A \emph{subshift} on $\sf{A}$ is a closed subset $X\subset\sf{A}^\bb{Z}$ with $\varphi(X)=X$. Note that for a subshift $X$ we have an induced action $\bb{Z}\acts X$ implemented by the restriction of $\varphi$ to $X$. We say that $X$ is a \emph{minimal subshift} when it is a subshift such that the induced $\bb{Z}$ action by $\varphi$ is minimal (which is then also automatically free). Minimal subshifts that are infinite are Cantor spaces \cite{Bru22}. A subshift $X$ is minimal if and only if $\mc{L}(X)=\mc{L}(x)$ for all $x\in X$.

We obtain an action of the infinite dihedral group $\bb{Z}\rtimes\bb{Z}_2$ on $X$ from $\varphi$ together with a homeomorphism $\sigma\in\mathrm{Homeo}(X)$ satisfying $\sigma^2=\mathrm{id}$ and $\sigma\varphi\sigma=\varphi^{-1}$ (i.e.\ $\sigma$ \emph{anti-commutes} with $\varphi$), so that the homeomorphism corresponding to $(n,j)\in\bb{Z}\rtimes\bb{Z}_2$ is $\varphi^n\sigma^j$, and we write $(\varphi,\sigma)\colon\bb{Z}\rtimes\bb{Z}_2\acts X$.

We now consider $\mathsf{A}=\{\mathsf{0},\mathsf{1}\}$, and let $\{\mathsf{0},\mathsf{1}\}^*\ni v\mapsto \hat{v}\in\{\sf{0},\sf{1}\}^*$ be the map given by reversing and ``bit-swapping", i.e.\ if $v=x_1\dots x_{n}$ with $x_i\in\{\sf{0},\sf{1}\}$ for $i=1,\dots,n$, then $\hat{v} = (\mathsf{1}-x_{n})\dots(\mathsf{1}-x_1)$. We refer to the map $v\mapsto\hat{v}$ as \emph{anti-reversal}. A finite word $v\in\{\sf{0},\sf{1}\}^*$ is called an \emph{anti-palindrome} when it is a fixed point of anti-reversal, i.e.\ $v=\hat{v}$. We now define a specific $\sigma\in\mathrm{Homeo}(\{\mathsf{0},\mathsf{1}\}^\bb{Z})$ as
\begin{equation}
\sigma(x_j)_{j\in\bb{Z}}=(\mathsf{1} - x_{-j})_{j\in\bb{Z}},\quad (x_j)_{j\in\bb{Z}}\in\{\mathsf{0},\mathsf{1}\}^\bb{Z}.
\end{equation}
It is directly seen that $\sigma$ is a free involution that anti-commutes with the shift map $\varphi$. As for $\varphi\sigma$, the set of its fixed points is described as
\begin{equation}\label{eq:fix-phi-sigma}
\big\{\dots(\mathsf{1}-x_2)(\mathsf{1}-x_1)\; . \;x_1x_{2}\dots \colon  x_i\in\{\mathsf{0},\mathsf{1}\}\text{ for all }i\ge1 \big\},
\end{equation}
where the ``$.$" in \eqref{eq:fix-phi-sigma} is placed between slots $0$ and $1$.
Indeed, for $x=(x_j)_{j\in\bb{Z}}\in\{\mathsf{0},\mathsf{1}\}^\bb{Z}$ we have
\begin{align*}
\varphi\sigma(x_j)_{j\in\bb{Z}} &= \varphi(\mathsf{1}-x_{-j})_{j\in\bb{Z}}\\
&= (\mathsf{1}-x_{-j+1})_{j\in\bb{Z}}
\end{align*}
and so $x=\varphi\sigma(x)$ if and only if $x_j=\mathsf{1}-x_{-j+1}$ for all $j\in\bb{Z}$. We call such bi-infinite sequences \emph{anti-palindromic}.

\begin{lemma}\label{lem:good-subshifts}
Let $X \subset \{\mathsf{0},\mathsf{1}\}^\bb{Z}$ be a subshift and $\sigma$ the free involution defined above.
\begin{enumerate}[label=\normalfont{(\roman*)}]
\item\label{it:gs1} $\sigma(X)=X$ if and only if $\mc{L}(X)$ is closed under anti-reversal, i.e.\ $\widehat{\mc{L}(X)} = \mc{L}(X)$.
\item\label{it:gs2} $X$ contains a fixed point of $\varphi\sigma$ if and only if $\mc{L}(X)$ contains anti-palindromes of arbitrarily large length, i.e.\
\begin{equation*}
\sup\{|w|: w\in\mc{L}(X) \text{ is an anti-palindrome}\} = \infty.
\end{equation*}
In particular, if $X$ is a minimal subshift with $\sigma(X)=X$, then the action $(\varphi,\sigma)\colon \bb{Z}\rtimes\bb{Z}_2\acts X$ is free if and only if $\mc{L}(X)$ contains finitely many anti-palindromes.
\end{enumerate} 
\end{lemma}

\begin{proof}
\ref{it:gs1} Assume that $\widehat{\mc{L}(X)}=\mc{L}(X)$ and let $x\in X$. Let $N\in\bb{N}$ and set $v\coloneqq x_{[-N,N]}$. Since $\hat{v}\in\mc{L}(X)$, there is some $y\in X$ such that $\hat{v}\prec y$, and by replacing $y$ with $\varphi^k(y)$ for some $k\in\bb{Z}$ if necessary, we can assume without loss of generality that $y_{[-N,N]}=\hat{v}$. Since $\sigma(x)_{[-N,N]} = \hat{v}$, we have that $y$ approximates $\sigma(x)$ at a tolerance of $2^{-N}$. As $N\in\bb{N}$ was arbitrary and $X$ is closed, this shows that $\sigma(x)\in X$. The reverse implication is immediate.

\ref{it:gs2} Suppose that $\mc{L}(X)$ contains anti-palindromes of arbitrarily large length, say $(w_n)_{n=1}^\infty$ with $|w_n|\nearrow \infty$. Since anti-palindromes necessarily have even length, we write $|w_n|=2\ell_n$ for $n\ge1$. For each $n\ge1$ we thus obtain $x^{(n)} \in X$ such that $w_n\prec x^{(n)}$, and by shifting $x^{(n)}$ appropriately we can assume without loss of generality that $x^{(n)}_{[-\ell_n+1,\ell_n]}=w_n$. By compactness the sequence $(x^{(n)})_{n=1}^\infty$ has a cluster point $x\in X$, and by passing to a subsequence if necessary we can assume that $x^{(n)}\to x$. Let $j\ge1$ and pick $n\in\bb{N}$ large enough so that $\ell_n\ge j$ and $x_{[-j+1,j]}^{(n)}=x_{[-j+1,j]}$. Since 
\begin{align*}
w_n&=x_{-\ell_n+1}^{(n)}x_{-\ell_n+2}^{(n)}\dots x_{-1}^{(n)}x_0^{(n)}x_1^{(n)}\dots x_{\ell_n}^{(n)} \\
&= x_{-\ell_n+1}x_{-\ell_n+2}\dots x_{-1}x_0x_1\dots x_{\ell_n}
\end{align*}
 is an anti-palindrome, $x_i=\mathsf{1}-x_{-i+1}$ for all $-\ell_n+1\le i\le \ell_n$, and in particular $x_j=\mathsf{1}-x_{-j+1}$. As $j\ge1$ was arbitrary, we conclude that $x$ is an anti-palindromic sequence. 
 
 The ``only if" part of the statement is immediate.
 
  Finally, the last statement follows from \cite[Proposition~2.8]{OrtSca23}.
\end{proof}

Let $\underline{t} \in \{\sf{0},\sf{1}\}^\bb{N}$ be the so-called \emph{paper-folding sequence}, also known as the \emph{dragon curve sequence}, defined as the limit of the recursive rule $t_0=\sf{1}$, and $t_{n+1}=t_n \mathsf{1} \hat{t}_n$ for $n\in\bb{N}$: since $t_n$ is an initial segment of $t_{n+1}$ for each $n\ge1$ and $|t_n|=2^{n+1}-1\to\infty$, $\underline{t}$ is well-defined (note that $t_n = \underline{t}_{[0,2^{n+1}-2]}$ for all $n\in\bb{N}$). The words $t_0,\dots,t_5$ are listed below.
\begin{align*}
&\mathsf{1} \\
&\mathsf{110} \\
&\mathsf{1101100}\\
&\mathsf{110110011100100}\\
&\mathsf{1101100111001001110110001100100}\\
&\mathsf{110110011100100111011000110010011101100111001000110110001100100}
\end{align*}

Arguing similarly as in \cite{BerBoaCarFag09}, we prove the following facts about the paper-folding sequence.
\begin{proposition}\label{prop:pf}
Let $\underline{t}\in\{\mathsf{0},\mathsf{1}\}^\bb{N}$ be the paper-folding sequence. Then
\begin{enumerate}[label=\normalfont{(\roman*)}]
\item\label{it1pf} $\widehat{\mc{L}(\underline{t})}=\mc{L}(\underline{t})$, i.e.\ $v\in\mc{L}(\underline{t})$ if and only if $\hat{v}\in\mc{L}(\underline{t})$,
\item\label{it2pf} $\underline{t}$ is uniformly recurrent and not periodic, and
\item\label{it3pf} if $v\in\mc{L}(\underline{t})$ is an anti-palindrome, then $|v|\le 6$. In particular, $\mc{L}(\underline{t})$ contains finitely many anti-palindromes.
\end{enumerate}
\end{proposition}

\begin{proof}
\ref{it1pf} If $v\in\mc{L}(\underline{t})$, then $v\prec t_n$ for some $n\in\bb{N}$, and so $\hat{v}\prec \hat{t}_n\prec t_{n+1}\prec \underline{t}$.

\ref{it2pf} We claim that for $p,n\in\mathbb{N}$, if $t_n=x_1\dots x_{2^{n+1}-1}$ ($x_i\in\{\mathsf{0},\mathsf{1}\}$), then
\begin{equation}\label{eq:pf-self-similarity}
t_{p+n+1} = t_p x_1 \hat{t}_p x_2 t_p x_3 \hat{t}_p\dots x_{2^{n+1}-2} t_p x_{2^{n+1}-1}\hat{t}_p.
\end{equation}
Fix $p\in\mathbb{N}$; we shall prove the claim by induction on $n$. For $n=0$, this is clear by definition, since $t_{p+1} = t_p\mathsf{1}\hat{t}_p$ (and $\mathsf{1}=t_0$). Assume that the claim has been proved for some $n\in\bb{N}$, so 
\[
t_{p+n+1} = t_p x_1 \hat{t}_p x_2 t_p x_3 \hat{t}_p\dots x_{2^{n+1}-2} t_p x_{2^{n+1}-1}\hat{t}_p
\]
with $t_n = x_1\dots x_{2^{n+1}-1}$. Then, since $t_{p+n+2} = t_{p+n+1} \mathsf{1} \hat{t}_{p+n+1}$,
\begin{align*}
t_{p+n+2} &= t_px_1\hat{t}_px_2t_p\dots t_px_{2^{n+1}-1}\hat{t}_p\mathsf{1} t_p \hat{x}_{2^{n+1}-1} \hat{t}_p \dots \hat{t}_p \hat{x}_2 t_p \hat{x}_1 \hat{t}_p
\end{align*}
and note that 
\begin{align*}
t_{n+1}&=t_{n}\mathsf{1}\hat{t}_{n} \\
&=x_1x_2\dots x_{2^{n+1}-1}\mathsf{1} \hat{x}_{2^{n+1}-1}\dots\hat{x}_2\hat{x}_1
\end{align*}
as we wanted, and so \eqref{eq:pf-self-similarity} holds.

To see that $\underline{t}$ is uniformly recurrent, let $v\prec \underline{t}$ be a finite subword. By definition, there is $p\in\bb{N}$ such that $v\prec t_p$. Let $w\prec \underline{t}$ be any subword with $|w|=3\cdot 2^{p+1}$. Let $N\ge p+2$ be such that $w\prec t_N$ and write $N=p+n+1$ for some $n\ge1$. Note that by \eqref{eq:pf-self-similarity} and since $|w|=3\cdot 2^{p+1}$, we must have that $t_p\prec w$ and thus $v\prec w$ as we wanted.

 Assume now that $\underline{t}$ is periodic, i.e.\ $\underline{t}=wvvv\dots$ for some $w,v\in\{\mathsf{0,1}\}^*$. By replacing $v$ with $vv$, we can assume without loss of generality that the length of $v$ is even. Let $c,d\ge0$ be such that $|v| = 2^{c+1}(2d+1)$. By replacing $v$ with some cyclic permutation of its letters if necessary (and thus without affecting its length), we can assume without loss of generality that $w=t_{cq}$ for some $q\ge1$. By \eqref{eq:pf-self-similarity} we have that
\begin{align*}
t_{cq+d+1} &= t_{c+ c(q-1)+d+1} \\
&=t_c x_1 \hat{t}_c \dots t_c x_{2^{c(q-1)}-1}\hat{t}_c  x_{2^{c(q-1)}} t_c x_{2^{c(q-1)}+1} \hat{t}_c \dots t_c x_{2^{c(q-1)+d+1}-1}\hat{t}_c,
\end{align*}
with $x_1\dots x_{2^{c(q-1)+d+1}-1}=t_{c(q-1)+d}$. Note that $t_{cq+d+1}$ is an initial segment of $\underline{t}$, and its initial segment 
\[
t_cx_1\hat{t}_c \dots t_cx_{2^{c(q-1)}-1}\hat{t}_c
\]
has length $2^{cq+1}-1$, and so is equal to $t_{cq}=w$. Therefore, the periodic part $vvv\dots$ has $x_{2^{c(q-1)}} t_c x_{2^{c(q-1)}+1} \hat{t}_c \dots t_c x_{2^{c(q-1)+d+1}-1}\hat{t}_c $ as an initial segment, which in turn has $x_{2^{c(q-1)}}t_cx_{2^{c(q-1)}+1} \hat{t}_cx_{2^{c(q-1)}+2}\dots \hat{t}_cx_{2^{c(q-1)}+2d} t_c$ as an initial segment of length $2^{c+1}(2d+1)=|v|$, so 
\[
v=x_{2^{c(q-1)}}t_cx_{2^{c(q-1)}+1} \hat{t}_cx_{2^{c(q-1)}+2}\dots \hat{t}_cx_{2^{c(q-1)}+2d} t_c.
\]
This is succeeded by $x_{2^{c(q-1)}+2d+1}\hat{t}_c$ which is an initial segment of $v$, and so it follows that $t_c=\hat{t}_c$. This is a contradiction, as this would imply that $|t_c|_\mathsf{0}=|t_c|_\mathsf{1}$, which is never the case since $|t_c|_\mathsf{1}=|t_c|_\mathsf{0}+1$ by construction.

\ref{it3pf} Note first that any anti-palindrome has even length. Now if $v=v_1\dots v_{2n}\in\{\mathsf{0},\mathsf{1}\}^*$ is an anti-palindrome of length $2n>8$, then 
\[
v_{n-3}v_{n-2}v_{n-1}v_nv_{n+1}v_{n+2}v_{n+3}v_{n+4}
\]
is an anti-palindrome of length $8$ that is a subword of $v$. It thus suffices to show that $\underline{t}$ contains no anti-palindromes of length $8$, i.e.\ that none of the $t_p$ contain anti-palindromes of length $8$. We directly see that
\[
t_4 = \mathsf{1101100111001001110110001100100}
\]
contains no anti-palindromes of length $8$. We continue with induction: assume we have shown that $t_p$ contains no anti-palindromes of length $8$ for some $p\ge 4$. By \eqref{eq:pf-self-similarity} we have $t_{p} = t_3 s \hat{t}_3$ for some $s\prec \underline{t}$ (with $|s|= 2^{p+1}-2^5+1$), and $t_{p+1} = t_3 s \hat{t}_3 \mathsf{1} t_3 \hat{s} \hat{t}_3$. Since $|t_3| = 2^4-1=15$, if an anti-palindrome of length $8$ occurs as a subword of $t_{p+1}$, then it is either a subword of $t_3s\hat{t}_3=t_p$, or a subword of $\hat{t}_3\mathsf{1} t_3$, or a subword of $t_3\hat{s}\hat{t}_3=\hat{t}_p$. The first case and the third case are both impossible due to the inductive hypothesis (and upon noting that if $v\prec \hat{t}_p$ is an anti-palindrome of length $8$ then $v=\hat{v}\prec t_p$ is an anti-palindrome of length $8$). As for the second case, we directly see that
\[
\hat{t}_3\mathsf{1}t_3 = \mathsf{110110001100100\;1\;110110011100100}
\]
contains no anti-palindromes of length $8$, and thus $t_{p+1}$ contains no anti-palindromes of length $8$. The proof is complete by induction.
\end{proof}

\begin{proposition}\label{lem:bi-infinite-t}
Let $\mathsf{A}$ be a finite alphabet and let $\underline{s}\in \mathsf{A}^\bb{N}$ be a non-periodic, uniformly recurrent infinite word. If
\[
X\coloneqq \{x\in\mathsf{A}^\bb{Z}: \mathcal{L}(x)\subset \mathcal{L}(\underline{s})\},
\]
then $X$ contains a bi-infinite sequence $\tilde{x}$ with $\tilde{x}_{[0,\infty)}=\underline{s}$. Moreover, $X = \{x\in\mathsf{A}^\bb{Z}: \mathcal{L}(x) =  \mathcal{L}(\underline{s})\}$, and $X$ is a minimal subshift and a Cantor space.
\end{proposition}

\begin{proof}
To define such an $\tilde{x}=(\tilde{x}_j)_{j\in\bb{Z}}$, we first define $\tilde{x}_j$ for $j\ge0$ such that $\tilde{x}_{[0,\infty)}=s$. For $j\le -1$, we will define $\tilde{x}_j$ recursively so that $\tilde{x}_j\dots \tilde{x}_{j+n}\in\mathcal{L}(\underline{s})$ for all $n\in\bb{N}$. This suffices to conclude that $\mathcal{L}(\tilde{x})=\mathcal{L}(\underline{s})$. Let $j\le -1$ and assume that $\tilde{x}_{j+1},\tilde{x}_{j+2},\dots$ have all been defined so that
\[
\tilde{x}_{j+1}\dots \tilde{x}_{j+n}\in\mathcal{L}(\underline{s})
\]
for all $n\ge1$. We claim that there is an $a\in\mathsf{A}$ such that $a\tilde{x}_{j+1}\dots \tilde{x}_{j+n}\in\mathcal{L}(\underline{s})$ for all $n\ge1$. Indeed, if this was not the case, then for any $a\in \mathsf{A}$ there is some $n_a\ge1$ such that $a\tilde{x}_{j+1}\dots \tilde{x}_{j+n_a}\not\in\mathcal{L}(\underline{s})$. Setting $n\coloneqq \max\{n_a:a\in\mathsf{A}\}$, we conclude that $a\tilde{x}_{j+1}\dots \tilde{x}_{j+n}\not\in\mathcal{L}(\underline{s})$ for all $a\in\mathsf{A}$. However, since $\underline{s}$ is (uniformly) recurrent and $\tilde{x}_{j+1}\dots \tilde{x}_{j+n}$ is a subword of $\underline{s}$, this finite subword occurs in $\underline{s}$ preceded by some letter $a\in\mathsf{A}$, and so $a\tilde{x}_{j+1}\dots \tilde{x}_{j+n}\in\mathcal{L}(\underline{s})$, which is a contradiction.

It follows from the preceding paragraph that $X$ is nonempty, and it is actually infinite, since $\{\varphi^k(\tilde{x})\}_{k\ge0}$ is infinite, due to $\underline{s}$ being non-periodic. Note that $\varphi(X)=X$ since $\mc{L}(\varphi(x))=\mc{L}(x)$ for any $x\in\mathsf{A}^\bb{Z}$. That $X$ is closed follows from the fact that $\underline{s}$ is uniformly recurrent: let $(x^{(m)})_{m=1}^\infty\subset X$ be a sequence converging to $x\in\mathsf{A}^\bb{Z}$, let $v\prec x$ and let $N\in\bb{N}$ be such that $v\prec x_{[-N,N]}$. By taking $m$ large enough, $x^{(m)}$ agrees with $x$ on the slots $[-N,N]$ and so $v\prec x^{(m)}$, hence $v\in\mc{L}(x^{(m)})=\mc{L}(\underline{s})$, which shows that $\mc{L}(x)\subset\mc{L}(\underline{s})$ and so $x\in X$. 

Now if $x\in X$ and $v\in\mc{L}(\underline{s})$, since $\underline{s}$ is uniformly recurrent there is $\ell\in\bb{N}$ such that if $w\prec \underline{s}$ with $|w|\ge\ell$, then $v\prec w$. Since $x_{[1,\ell]}\in\mc{L}(x)\subset\mc{L}(\underline{s})$ has length $\ell$, we get that $v\prec x_{[1,\ell]}\prec x$, and thus $\mc{L}(\underline{s})\subset\mc{L}(x)$. This shows that $X = \{ x\in \mathsf{A}^\bb{Z}: \mc{L}(x)=\mc{L}(\underline{s})\}$, and thus $X$ is minimal since $\mc{L}(x)=\mc{L}(X)$ for all $x\in X$. In particular $X$ is a Cantor space as an infinite minimal subshift.
\end{proof}

\begin{definition}\label{def:pf-subshift}
With $\underline{t}\in\{\mathsf{0},\mathsf{1}\}^\bb{N}$ the paper-folding sequence, the \emph{paper-folding subshift} is defined as
\begin{align}\label{eq:pf-subshift}
\underline{X}&\coloneqq\{x\in\{\mathsf{0},\mathsf{1}\}^\bb{Z}: \mc{L}(x) \subset \mc{L}(\underline{t})\}\\
&\;=\{x\in\{\mathsf{0},\mathsf{1}\}^\bb{Z}: \mc{L}(x) = \mc{L}(\underline{t})\}.
\end{align}
\end{definition}

Combining \Cref{lem:good-subshifts}, \Cref{prop:pf} and \Cref{lem:bi-infinite-t}, we obtain the following corollary.

\begin{corollary}\label{cor:free-min-dihedral-pf}
The shift map $\varphi$ and the anti-reversal involution $\sigma$ restricted to the Cantor space $\underline{X}$ define a free minimal action $(\varphi,\sigma)\colon\bb{Z}\rtimes\bb{Z}_2\acts \underline{X}$.
\end{corollary}

\section{K-theory computations}

\noindent In this section we compute the $\mathrm{K}$-theory of the crossed product $\mathrm{C}^\ast$-algebra associated to the action of \Cref{cor:free-min-dihedral-pf}. We first compute (as a scaled ordered group) the  $\mathrm{K}_0$-group of the crossed product $\mathrm{C}^\ast$-algebra of the Cantor minimal system $\varphi\colon\bb{Z}\acts \underline{X}$ given by the shift map restricted to the paper-folding subshift of \Cref{def:pf-subshift}. The reader is referred to \cite{Bla06} for elements of $\mathrm{K}$-theory of $\mathrm{C}^\ast$-algebras as well as the relevant terminology and notation.

In order to compute $\mathrm{K}_0(C(\underline{X})\rtimes_\varphi\bb{Z})$ we define an auxiliary \emph{substitution subshift} over the four-letter alphabet $\mathsf{A}\coloneqq\{\mathsf{0},\mathsf{1},\mathsf{2},\mathsf{3}\}$. Consider the map $\varrho\colon \mathsf{A}^*\to\mathsf{A}^*$ defined on the generators (i.e.\ the letters) as
\begin{equation}\label{eq:rho}
\varrho\colon\begin{cases}\mathsf{3}\longrightarrow \mathsf{31} \\ \mathsf{2}\longrightarrow \mathsf{30} \\ \mathsf{1}\longrightarrow \mathsf{21}\\ \mathsf{0}\longrightarrow \mathsf{20}\end{cases}
\end{equation}
and extended on $\mathsf{A}^*$ as $\varrho(x_1\dots x_n)\coloneqq\varrho(x_1)\dots\varrho(x_n)$ for all $n\ge1$ and $x_1,\dots,x_n\in\mathsf{A}$. Setting
\begin{equation}
\mc{L}_\varrho \coloneqq \bigcup_{n\ge1}\bigcup_{a\in\mathsf{A}}\mc{L}(\varrho^n(a)),
\end{equation}
we consider the subshift (cf.\ \cite[\S 3.3.1]{DurHosSka99})
\begin{equation}
X_\varrho \coloneqq \big\{ y\in\mathsf{A}^\bb{Z}: \mc{L}(y) \subset \mc{L}_\varrho \big\}.
\end{equation}
The map $\varrho$ is a \emph{primitive substitution}, in the sense that there is $n\in\bb{N}$ such that for any seed letter $a\in\mathsf{A}$, every letter of $\sf{A}$ appears in the iteration $\varrho^n(a)$ (in our case one can take $n=3$), and the length of words arising as iterations of $\varrho$ increases to infinity. It is well-known that primitive substitutions give rise to minimal and \emph{uniquely ergodic} systems (cf.\ \cite[\S 3.3.1]{DurHosSka99}) where the latter means that there is a unique invariant Borel probability measure on the space. In particular, $\varphi\colon \bb{Z}\acts X_\varrho$ is minimal and uniquely ergodic.

\begin{remark}\label{rmk:r}
Since $\varrho$ is a primitive substitution and since $\varrho^n(a)$ begins with the same letter (namely $\mathsf{3}$) for all $a\in\mathsf{A}$ and $n\ge2$, the iterations of $\varrho$ (independently of the seed letter) converge to an infinite word $\underline{r}\in\mathsf{A}^\bb{N}$ (cf.\ \cite[Section~1.2]{Fog02}), the first few digits of which read as follows:
\begin{equation*}
\underline{r}=\sf{31213021312030213121302031203021}\dots
\end{equation*}
By definition, $\mc{L}_\varrho=\mc{L}(\underline{r})$. The word $\underline{r}$ is closely related to the paper-folding sequence $\underline{t}\in\{\mathsf{0},\mathsf{1}\}^\bb{N}$ -- see \cite[Example~4.10, p.\ 138]{Bru22}. In particular, observe that starting from the $0$th slot of $\underline{t}$ and replacing two-letter blocks according to the rule $xy\mapsto \mathsf{2}x+y$ (i.e.\ $\mathsf{00}\mapsto\mathsf{0}$, $\mathsf{01}\mapsto\mathsf{1}$, $\mathsf{10}\mapsto \mathsf{2}$ and $\mathsf{11}\mapsto\mathsf{3}$), one obtains $\underline{r}$. It follows from this observation that, since $\underline{t}$ is uniformly recurrent and non-periodic, also $\underline{r}$ is uniformly recurrent and non-periodic.
\end{remark}

According to \cite[\S 6.4]{DurHosSka99}, we associate to $X_\varrho$ the $4\times 4$ matrix $M$ with non-negative integer entries, indexed over $\{\mathsf{0},\mathsf{1},\mathsf{2},\mathsf{3}\}$, where the $(i,j)$ entry of $M$ is defined as the number of occurrences of the letter $j$ in the word $\varrho(i)$. Inspecting \eqref{eq:rho}, we see that
\begin{equation}\label{eq:matrix-of-rho}
M = \begin{pmatrix} 1 & 0 & 1 & 0 \\ 0 &1 & 1 & 0 \\ 1 & 0 & 0 & 1 \\ 0 & 1 &0 & 1 \end{pmatrix}.
\end{equation} 

\begin{proposition}\label{prop:K-theory-rho}
With $\varphi\colon\bb{Z}\acts X_\varrho$ the substitution subshift defined above,  $\mathrm{K}_0(C(X_\varrho)\rtimes_\varphi\bb{Z})$ is isomorphic as a scaled ordered group to
\begin{equation}\label{eq:dyadic-oplus-z}
\big(\bb{Z}[\tfrac{1}{2}]\oplus\bb{Z}, \{(s,n): s>0,n\in\bb{Z}\}\cup\{(0,0)\},(1,0)\big).
\end{equation}
\end{proposition}

\begin{proof}
Note first that the substitution $\varrho$ is \emph{left-proper} (cf.\ \cite[Definition~16]{DurHosSka99}), in the sense that there is $p\in\mathbb{N}$ so that the words $\varrho^p(a)$ start with the same letter for all $a\in\mathsf{A}$ -- in our case one sees that for $p=2$ all words $\varrho^2(\mathsf{0})=\mathsf{3020},\varrho^2(\mathsf{1})=\mathsf{3021},\varrho^2(\mathsf{2})=\mathsf{3120}$, and $\varrho^2(\mathsf{3})=\mathsf{3121}$ start with the letter $\mathsf{3}$. We can thus apply \cite[Corollary~33]{DurHosSka99} (together with \cite[\S 2.4.1, Remark~(ii)]{DurHosSka99}) and so $\mathrm{K}_0(C(X_\varrho)\rtimes_\varphi\bb{Z})$ is computed as an inductive limit: with $M$ the matrix of \eqref{eq:matrix-of-rho}, for $n\ge 1$ consider the subgroups of $\bb{Q}^4$ given by
\begin{equation}
G_n \coloneqq \big\{\vec{q}\in \bb{Q}^4: M^n\cdot \vec{q} \in \bb{Z}^4\big\},
\end{equation}
\begin{equation}
H_n\coloneqq \big\{\vec{q}\in\bb{Q}^4: M^n\cdot\vec{q} = \vec{0}\big\},
\end{equation}
and the subsets
\begin{equation}
(G_{n})_+\coloneqq\big\{\vec{q}\in \bb{Q}^4: M^n\cdot\vec{q}\in\bb{Z}_{\ge0}^4\big\}.
\end{equation}
Note that $G_n\subset G_{n+1}$, $H_n\subset G_n$, $H_n \subset H_{n+1}$ and  that $(G_{n})_+\subset (G_{n+1})_+$ for all $n\ge1$. Set $G\coloneqq \bigcup_{n\ge1}G_n$, $H\coloneqq \bigcup_{n\ge1}H_n \subset G$ and $G_+\coloneqq \bigcup_{n\ge1}(G_n)_+$. Then, by \cite[\S 6.4 and Corollary~33]{DurHosSka99} we have that $\mathrm{K}_0(C(X_\varrho)\rtimes_\varphi\bb{Z})$ is isomorphic as an abelian group to $G/H$, its positive cone is identified with the image of $G_+$ under the quotient map, and the order unit is identified with the equivalence class of $(1,1,1,1)$ in $G/H$. By induction one checks that
\begin{equation}
M^{n+2}=\begin{pmatrix} 2^n+1 & 2^n-1 & 2^n & 2^n \\ 2^n & 2^n & 2^n & 2^n \\ 2^n & 2^n & 2^n & 2^n \\ 2^n-1 & 2^n+1 & 2^n &2^n\end{pmatrix}, \quad n\in\bb{N},
\end{equation}
and so with the notation $\vec{q}=(q_1,\dots,q_4)\in\bb{Q}^4$ we obtain
\begin{equation}
G_{n+2}= \big\{\vec{q}\in\bb{Q}^4: 2^n(q_1+\dots+q_4) \in \bb{Z},\; q_1-q_2\in\bb{Z}\big\},
\end{equation}
\begin{equation}
H_{n+2} = \big\{\vec{q}\in\bb{Q}^4: q_1+\dots+q_4=0,\; q_1=q_2\big\},
\end{equation}
and
\begin{equation}
(G_{n+2})_+ = \big\{ \vec{q}\in\bb{Q}^4: 2^n(q_1+\dots+q_4)\in\bb{N},\;|q_1-q_2|\le2^n(q_1+\dots+q_4)\big\}.
\end{equation}
We thus define group homomorphisms $\alpha_n\colon G_{n+2}\to \tfrac{1}{2^n}\bb{Z}\oplus\bb{Z}$ given by $\alpha_n(\vec{q})=(q_1+\dots+q_4,q_1-q_2)$. It is clear that each $\alpha_n$ is surjective, that $\ker(\alpha_n)=H_{n+2}$ and that 
\[
\alpha_n((G_{n+2})_+) = \{(s,m)\in\tfrac{1}{2^n}\bb{Z}\oplus \bb{Z}: s\ge 0, |m|\le 2^ns\}.
\]
Denoting by $\tilde{\alpha}_n\colon G_{n+2}/H_{n+2}\to\tfrac{1}{2^n}\bb{Z}\oplus\bb{Z}$ the induced isomorphisms, we see that we have a commutative diagram
\begin{equation}
\begin{tikzcd}
G_2/H_2\ar[r, hook]\ar[d, "\tilde{\alpha}_0", "\cong" left] & G_3/H_3\ar[d, "\tilde{\alpha}_1", "\cong" left] \ar[r,hook] & G_4/H_4\ar[d, "\tilde{\alpha}_2", "\cong" left] \ar[r,hook] & \dots \\
\tfrac{1}{2^0}\bb{Z}\oplus \bb{Z} \ar[r,hook]&\tfrac{1}{2^1}\bb{Z}\oplus\bb{Z} \ar[r,hook] & \tfrac{1}{2^2}\bb{Z}\oplus\bb{Z}\ar[r, hook] &\dots
\end{tikzcd}
\end{equation}
where the arrows on each row denote the inclusion maps. The abelian group that is the inductive limit of the bottom row is $\bb{Z}[\tfrac{1}{2}]\oplus\bb{Z}$, so we obtain an isomorphism $G/H\cong\bb{Z}[\tfrac{1}{2}]\oplus\bb{Z}$. Moreover, the positive cone of $G/H$ (being the image of $G_+$ under the quotient map) is carried by this isomorphism to $\bigcup_{n\in\bb{N}}\tilde{\alpha}_n((G_{n+2})_+)$ which is precisely the set 
\begin{equation}\label{eq:positive-cone}
\bigcup_{n\in\bb{N}}\{(s,m)\in\tfrac{1}{2^n}\bb{Z}\oplus \bb{Z}: s\ge0, |m|\le 2^n s\}  ,
\end{equation}
which is in fact equal to $\{(s,m)\in\bb{Z}[\tfrac{1}{2}]\oplus\bb{Z}: s>0\}\cup\{(0,0)\}$: the latter contains the set appearing in \eqref{eq:positive-cone} which obviously contains $(0,0)$. On the other hand, if $(s,m)\in\bb{Z}[\tfrac{1}{2}]\oplus\bb{Z}$ is such that $s>0$, then we can write $s=\frac{k}{2^n}$ for some $k,n\in\bb{N}$. Now take $n'\in\bb{N}$ large enough so that $|m|\le 2^{n'}k $ and note that $s =\frac{2^{n'}k}{2^{n+n'}} \in \tfrac{1}{2^{n+n'}}\bb{Z}$, and $2^{n+n'}s=2^{n'}k$.

Finally, the isomorphism $G/H\cong\bb{Z}[\tfrac{1}{2}]\oplus\bb{Z}$ maps the equivalence class of $(1,1,1,1)$ to the element $(4,0)$. This shows that, as a scaled ordered group, $\mathrm{K}_0(C(X_\varrho)\rtimes_\varphi\bb{Z})$ is isomorphic to
\[
\big(\bb{Z}[\tfrac{1}{2}]\oplus\bb{Z}, \{(s,n): s>0,n\in\bb{Z}\}\cup\{(0,0)\},(4,0)\big)
\]
which is in turn isomorphic to the scaled ordered group in \eqref{eq:dyadic-oplus-z}, since $\bb{Z}[\tfrac{1}{2}]$ is $2$-divisible.
\end{proof}

In order to make use of our auxiliary subshift we need a transformation from the binary shift space. Consider the map $\beta\colon\{\mathsf{0},\mathsf{1}\}^\bb{Z}\to\mathsf{A}^\bb{Z}$ (where still $\mathsf{A}=\{\mathsf{0,1,2,3}\}$) defined as
\begin{equation}\label{eq:beta}
\beta\big( (x_i)_{i\in\bb{Z}} \big) = (\mathsf{2}\cdot x_{2i} + x_{2i+1})_{i\in\bb{Z}}, \quad (x_i)_{i\in\bb{Z}}\in\{\mathsf{0},\mathsf{1}\}^\bb{Z}
\end{equation}
and note that $\beta$ is injective. While $\beta$ is not equivariant with respect to the shift maps, one directly checks that
\begin{equation}\label{eq:beta-shiftsquare}
\beta\circ\varphi_{\{\mathsf{0},\mathsf{1}\}}^2 = \varphi_{\mathsf{A}}\circ\beta.
\end{equation}

\begin{lemma}\label{lem:t-to-r}
Let $\beta$ be the map defined in \eqref{eq:beta}. Then, there is $\tilde{x}\in \underline{X}$ such that $\tilde{x}_{[0,\infty)}=\underline{t}$ and $\beta(\tilde{x})\in X_\varrho$, with $\beta(\tilde{x})_{[0,\infty)}=\underline{r}$, where still $\underline{t}$ is the paper-folding sequence from \Cref{sec:pf} and $\underline{r}$ comes from \Cref{rmk:r}. Moreover, with 
\begin{equation}
\underline{X}^{(0)}\coloneqq \overline{\{\varphi_{\{\mathsf{0,1}\}}^{2k}(\tilde{x}): k\in\bb{Z}\}}
\end{equation}
and
\begin{equation}
\underline{X}^{(1)}\coloneqq \overline{\{\varphi_{\{\mathsf{0,1}\}}^{2k+1}(\tilde{x}): k\in\bb{Z}\}},
\end{equation}

\begin{enumerate}[label=\normalfont{(\roman*)}]
\item\label{it1decomp} $\underline{X}^{(0)}$ and $\underline{X}^{(1)}$ are both mapped onto themselves by $\varphi_{\{\mathsf{0,1}\}}^2$, and they are mapped onto each other by $\varphi_{\{\mathsf{0,1}\}}$,
\item\label{it2decomp} $\underline{X}^{(0)}\sqcup \underline{X}^{(1)}=\underline{X}$ (and so $\underline{X}^{(0)}$ and $\underline{X}^{(1)}$ are clopen sets),
\item\label{it3decomp} $\beta(\underline{X}^{(0)}) = X_\varrho$.
\end{enumerate}

\end{lemma}

\begin{proof}
For the first part of the statement, from \Cref{lem:bi-infinite-t} we obtain $y\in X_\varrho \subset \mathsf{A}^\bb{Z}$ such that $y_{[0,\infty)}=\underline{r}$. Consider $\tilde{x}\in\{\mathsf{0},\mathsf{1}\}^\bb{Z}$ defined as $\tilde{x}_{[0,\infty)}=\underline{t}$, and, for $j\le -1$, we choose $\tilde{x}_{2j},\tilde{x}_{2j+1}\in\{\mathsf{0},\mathsf{1}\}$ in the unique way so that $\mathsf{2}\tilde{x}_{2j}+\tilde{x}_{2j+1}=y_j$. By definition of $\beta$ together with \Cref{rmk:r}, $\beta(\tilde{x})=y$, so what we need to show is that $\tilde{x}\in \underline{X}$. Indeed, let $k,\ell\in\bb{Z}$ with $k\le \ell$ and consider the subword $\tilde{x}_{2k}\tilde{x}_{2k+1}\dots\tilde{x}_{2\ell}\tilde{x}_{2\ell+1}$ of $\tilde{x}$. Since $\beta(\tilde{x})=y$ and $\mathcal{L}(y)=\mathcal{L}(\underline{r})$, we have that $w\coloneqq(\mathsf{2}\tilde{x}_{2k}+\tilde{x}_{2k+1})\dots(\mathsf{2}\tilde{x}_{2\ell}+\tilde{x}_{2\ell+1})$ is a subword of $\underline{r}=(\mathsf{2}\tilde{x}_0+\tilde{x}_1)(\mathsf{2}\tilde{x}_2+\tilde{x}_3)\dots$, and so there exists $m\ge0$ such that 
\[
w = (\mathsf{2}\tilde{x}_{2m}+\tilde{x}_{2m+1})\dots(\mathsf{2}\tilde{x}_{2m+2\ell-2k} +\tilde{x}_{2m+2\ell-2k+1}).
\]
Since $\mathsf{2}a+b=\mathsf{2}c+d$ implies that $a=b$ and $c=d$ for all $a,b,c,d\in\{\mathsf{0},\mathsf{1}\}$, it follows that
\[
\tilde{x}_{2k}\tilde{x}_{2k+1}\dots\tilde{x}_{2\ell}\tilde{x}_{2\ell+1} = \tilde{x}_{2m}\tilde{x}_{2m+1}\dots\tilde{x}_{2m+2\ell-2k}\tilde{x}_{2m+2\ell-2k+1},
\]
which is a subword of $\underline{t}=\tilde{x}_{[0,\infty)}$ as we wanted.

For the second part, \ref{it1decomp} is obvious. For \ref{it2decomp}, it is enough to show that $\underline{X}^{(0)} $ and $\underline{X}^{(1)}$ are disjoint since by minimality their union is equal to $\underline{X}$. To that end, observe first that $\underline{X}^{(0)} = \overline{\{\varphi_{\{\mathsf{0,1}\}}^{2k}(\tilde{x}): k\in\bb{N}\}}$ and that $\underline{X}^{(1)} = \overline{\{\varphi_{\{\mathsf{0,1}\}}^{2\ell+1}(\tilde{x}): \ell\in\bb{N}\}}$ (cf.\ \cite[Lemma~5.2]{Put18}). Now since $\tilde{x}_{[0,\infty)}=\underline{t}$, it is enough to show that for any $k,\ell\in\bb{N}$ we have that $\underline{t}_{[2k, 2k+6]} \ne \underline{t}_{[2\ell+1, 2\ell+7]}$, since this implies that
\begin{equation}
d(\varphi_{\{\mathsf{0,1}\}}^{2k}(\tilde{x}), \varphi_{\{\mathsf{0,1}\}}^{2\ell+1}(\tilde{x}))\ge 2^{-7}, \quad k,\ell\in\bb{N},
\end{equation}
where $d$ is the metric defined in \eqref{eq:metric}. Let $k,\ell\ge0$ and let $n\in\bb{N}$ be such that $\max\{2k+6, 2\ell+7\}<2^{n+3}-1$. With $t_\mathrm{e}\coloneqq \underline{t}_{[2k,2k+6]}$ and $t_\mathrm{o}\coloneqq \underline{t}_{[2\ell+1,2\ell+7]}$, we have that $t_\mathrm{e}, t_\mathrm{o}\prec t_{n+2}$. Applying \eqref{eq:pf-self-similarity} with $p=1$ and this $n$, we see that
\begin{equation}\label{eq:lol}
t_{n+2} = \mathsf{110} x_1 \mathsf{100} x_2 \mathsf{110} x_3 \mathsf{100} x_4 \mathsf{110}x_5 \dots \mathsf{110} x_{2^{n+1}-1}\mathsf{100}
\end{equation}
for some $x_i\in\{\mathsf{0,1}\}$. Inspecting \eqref{eq:lol}, we see that
\begin{equation}
t_\mathrm{e} \in \big\{ \mathsf{110}x\mathsf{100}, \mathsf{0}x\mathsf{100}y\mathsf{1}, \mathsf{100}x\mathsf{110}, \mathsf{0}x\mathsf{110}y\mathsf{1} \big\}
\end{equation}
for some $x,y\in\{\mathsf{0,1}\}$, while
\begin{equation}
t_\mathrm{o}\in \big\{ \mathsf{10}x\mathsf{100}y, x\mathsf{100}y\mathsf{11}, \mathsf{00}x\mathsf{110}y, x\mathsf{110}y\mathsf{10}\big\}
\end{equation}
for some $x,y\in\{\mathsf{0,1}\}$, and from this description it is clear that $t_{\mathrm{e}}\ne t_{\mathrm{o}}$.

For \ref{it3decomp}, notice that by \eqref{eq:beta-shiftsquare} we have that $\{\varphi_{\{\mathsf{0,1}\}}^{2k}(\tilde{x}):k\in\bb{Z}\}$ is mapped by $\beta$ onto the orbit of $\beta(\tilde{x})$ under the shift map on $\mathsf{A}^\bb{Z}$. Since $\beta(\tilde{x})$ belongs to $X_\varrho$ which is a minimal subshift, $\beta(\underline{X}^{(0)})=X_\varrho$.
\end{proof}

\begin{proposition}\label{thm:k_0-of-pf}
As a scaled ordered group, $\mathrm{K}_0(C(\underline{X})\rtimes_{\varphi_{\{\mathsf{0,1}\}}}\bb{Z})$ is isomorphic to
\begin{equation*}
\big(\bb{Z}[\tfrac{1}{2}]\oplus\bb{Z}, \{(s,n): s>0,n\in\bb{Z}\}\cup\{(0,0)\},(1,0)\big).
\end{equation*}
\end{proposition}

\begin{proof}
Set $A\coloneqq C(\underline{X})\rtimes_{\varphi_{\{\mathsf{0,1}\}}}\bb{Z}$, $B\coloneqq C(X_\varrho)\rtimes_{\varphi_\mathsf{A}}\bb{Z}$, and let $u_A\in A$, $u_B\in B$ be the unitaries that are implementing the respective $\bb{Z}$-actions, in the sense that $u_Afu_A^* = f\circ \varphi_{\{\mathsf{0,1}\}}^{-1}$ for all $f\in C(\underline{X})$ and $u_Bfu_B^* = f\circ\varphi_{\mathsf{A}}^{-1}$ for all $f\in C(X_\varrho)$. We will show that $A\cong \mathrm{M}_2(B)$, from which it follows that $\mathrm{K}_0(A) \cong \mathrm{K}_0(B)$ as ordered abelian groups, with $[1_A]_0$ identified with $2[1_B]_0$, and so by \Cref{prop:K-theory-rho} and $2$-divisibility of $\bb{Z}[\tfrac{1}{2}]$ it follows that $(\mathrm{K}_0(A),\mathrm{K}_0(A)_+,[1_A]_0)$ is isomorphic to
\[
\big(\bb{Z}[\tfrac{1}{2}]\oplus\bb{Z}, \{(s,n): s>0,n\in\bb{Z}\}\cup\{(0,0)\},(1,0)\big).
\]

Indeed, let $p\in C(\underline{X})$ be the indicator function of $\underline{X}^{(0)}$, and set $v\coloneqq pu_A^*\in A$. Note that $vv^* = p$ and that $v^*v = 1_A-p$, since $\varphi_{\{\mathsf{0,1}\}}(\underline{X}^{(0)})=\underline{X}^{(1)}=\underline{X}\setminus \underline{X}^{(0)}$. Sending $v$ and $v^*$ to the matrix units $e_{12}$ and $e_{21}$ respectively, $A\cong \mathrm{M}_2(pAp)$, and so it suffices to show that $B \cong pAp$. Let $\vartheta\colon C(X_\varrho)\to C(\underline{X}^{(0)})\cong pC(\underline{X})p\subset pAp$ be the $^*$-homomorphism
\[
\vartheta(f) = f\circ\beta\vert_{\underline{X}^{(0)}},\quad f\in C(X_\varrho)
\]
and let $\tilde{u}\coloneqq pu_A^2p$. Then 
\begin{align*}
\tilde{u}\vartheta(f)\tilde{u}^* &= f\circ \beta\vert_{\underline{X}^{(0)}}\circ \varphi_{\{\mathsf{0,1}\}}^{-2}\\
&\stackrel{\mathmakebox[\widthof{=}]{\eqref{eq:beta-shiftsquare}}}{=}\; (f\circ\varphi_{\mathsf{A}}^{-1})\circ\beta\vert_{\underline{X}^{(0)}} \\
&= \vartheta (u_B fu_B^*),
\end{align*}
and so we obtain a $^*$-homomorphism $C(X_\varrho)\rtimes_{\varphi_\mathsf{A}}\bb{Z}\to pAp $ that is injective since the domain is simple \cite{ArcSpi94}. Moreover it is surjective since the range clearly contains $pC(\underline{X})p$, and $\tilde{u}=pu_A^2p$, while $pu_Ap=0$.
\end{proof}

At this point the auxiliary subshift $X_\varrho$ has done its duty and it can go; this also means we can return to the alphabet $\{\mathsf{0},\mathsf{1}\}$ and write just $\varphi$ instead of $\varphi_{\{\mathsf{0},\mathsf{1}\}}$.

\begin{remark}\label{rmk:sigmastar-induces}
Regarding the free involution $\sigma$ on $\underline{X}$, observe that the map $\sigma_*\colon C(\underline{X},\bb{Z})\to C(\underline{X},\bb{Z})$, given by $f\mapsto f\circ\sigma^{-1}$, satisfies $\sigma_*(\mathrm{im}(1-\varphi_*)) \subset \mathrm{im}(1-\varphi_*)$ since $\sigma$ anti-commutes with $\varphi$. Clearly $\sigma_*$ maps the set of non-negative valued functions onto itself and the constant function $1$ to itself, so it follows that $\sigma_*$ induces an isomorphism of scaled ordered groups, which (with a slight abuse of notation) we also denote by $\sigma_*\colon \frac{C(\underline{X},\bb{Z})}{\mathrm{im}(1-\varphi_*)}\to \frac{C(\underline{X},\bb{Z})}{\mathrm{im}(1-\varphi_*)}$, and which moreover has order $2$, since $\sigma^2=\mathrm{id}$.
\end{remark}

The following lemma implies that the action induced by $\sigma$ on $\mathrm{K}_0(C(\underline{X})\rtimes_\varphi\bb{Z})$ is nontrivial. This also settles a problem raised by Scarparo in \cite[Remark~2.3~(ii)]{Sca23}.

\begin{lemma}
The automorphism $\sigma_*\colon \frac{C(\underline{X},\bb{Z})}{\mathrm{im}(1-\varphi_*)}\to\frac{C(\underline{X},\bb{Z})}{\mathrm{im}(1-\varphi_*)}$ is not the identity.
\end{lemma}

\begin{proof}
We will show that there is some $h\in C(\underline{X},\bb{Z})$ such that $h-\sigma_*(h)$ does not lie in $\mathrm{im}(1-\varphi_*)$, which is equivalent to $[h] \ne \sigma_*[h]$. Note that if $g\in\mathrm{im}(1-\varphi_*)$, then by writing $g=f-f\circ\varphi^{-1}$ for some $f\in C(\underline{X},\bb{Z})$ we have that
\begin{align}\label{eq:additional}
\bigg|\sum_{j=0}^ng(\varphi^j(x))\bigg|&=|f(\varphi^n(x))-f(\varphi^{-1}(x))| \nonumber\\
&\le 2\|f\|_\infty 
\end{align}
for all $x\in \underline{X}$ and all $n\in\bb{N}$, which is to say that
\begin{equation}\label{eq:gottschalk-hedlund}
\sup_{x\in \underline{X},n\in\bb{N}}\bigg|\sum_{j=0}^ng(\varphi^j(x))\bigg|<\infty.
\end{equation}
(Note that \eqref{eq:gottschalk-hedlund} is also a sufficient condition for $g$ to lie in $\mathrm{im}(1-\varphi_*)$ due to the Gottschalk--Hedlund theorem, but we shall make no use of this.)

Let $[\mathsf{1}]$ and $[\mathsf{0}]$ denote the (clopen) subsets of $\underline{X}$ of those $(x_j)_{j\in\bb{Z}}\in \underline{X}$ such that $x_0=\mathsf{1}$ and $x_0=\mathsf{0}$ respectively. Note that $\chi_{[\mathsf{1}]}-\sigma_*(\chi_{[\mathsf{1}]})=\chi_{[\mathsf{1}]}-\chi_{[\mathsf{0}]}$ and for the specific $g := \chi_{[\mathsf{1}]}-\chi_{[\mathsf{0}]} \in C(\underline{X},\bb{Z})$ observe that for $x\in \underline{X}$ and $n\in\bb{N}$ we have
\begin{equation}\label{eq:formula}
\sum_{j=0}^ng(\varphi^j(x))=|x_{[0,n]}|_\mathsf{1}-|x_{[0,n]}|_\mathsf{0}.
\end{equation}
Let $\tilde{x}\in \underline{X}$ be such that $\tilde{x}_{[0,\infty)}=\underline{t}$ (cf.\ \Cref{lem:bi-infinite-t}) and recursively define a sequence $(m_n)_{n=0}^\infty\subset\bb{N}$ by setting $m_0\coloneqq 0$ and 
\begin{equation}\label{eq:mn}
m_{n+1}\coloneqq\begin{cases} 2m_n+2,\quad \text{if } n \text{ is odd} \\ 2m_n+1,\quad \text{ if } n \text{ is even}. \end{cases}
\end{equation}
Note that $m_n$ is odd if and only if $n$ is odd (and so $m_n$ is even if and only if $n$ is even), and that $m_n\le 2^{n+1}-2$ for all $n\in\bb{N}$. We claim that
\begin{equation}\label{eq:claim}
|\underline{t}_{[0,m_n]} |_{\mathsf{1}}  - |\underline{t}_{[0,m_n]}|_{\mathsf{0}} = n+1,\quad n\in\bb{N}
\end{equation}
which we prove by induction. For $n=0$, the claim is directly seen to be true. Assume now that the claim is true for some even $n\in\bb{N}$ and let us prove it for $n+1$. Since $m_{n+1}\le 2^{n+2}-2$ and $|t_{n+1}|=2^{n+2}-1$, we have that $\underline{t}_{[0,m_{n+1}]}$ is an initial segment of $t_{n+1}$. By \eqref{eq:pf-self-similarity} (with $p=0$ therein)
\[
t_{n+1} = \mathsf{1} x_1 \mathsf{0} x_2 \mathsf{1} x_3 \mathsf{0} x_4 \mathsf{1} \dots \mathsf{0} x_{2^{n+1}-2} \mathsf{1} x_{2^{n+1}-1} \mathsf{0}
\]
with $x_1\dots x_{2^{n+1}-1}=t_{n}$, and so $x_1\dots x_{m_n}x_{m_n+1} = \underline{t}_{[0,m_n]}$. Since $m_{n+1}=2m_n+1$, and since $m_n$ is even,
\[
\underline{t}_{[0,m_{n+1}]} = \mathsf{1} x_1 \mathsf{0} x_2 \mathsf{1}x_3 \mathsf{0}x_4\dots \mathsf{0} x_{m_n-2} \mathsf{1} x_{m_n-1}\mathsf{0}x_{m_n}\mathsf{1}x_{m_n+1},
\]
whence 
\[
|\underline{t}_{[0,m_{n+1}]}|_\mathsf{1} =  |\underline{t}_{[0,m_n]}|_{\mathsf{1}} + (m_n+1)/2,
\] 
while 
\[
|\underline{t}_{[0,m_{n+1}]}|_\mathsf{0} = |\underline{t}_{[0,m_n]}|_\mathsf{0} + (m_n-1)/2
\]
and so 
\begin{align*}
|\underline{t}_{[0,m_{n+1}]}|_\mathsf{1}-|\underline{t}_{[0,m_{n+1}]}|_\mathsf{0} &= |\underline{t}_{[0,m_n]}|_\mathsf{1} -|\underline{t}_{[0,m_n]}|_\mathsf{0} +1 \\
& = n+2,
\end{align*}
as we wanted. A similar argument proves the claim for $n+1$ under the assumption that the claim is true for some odd $n\in\bb{N}$.

Since $\tilde{x}_{[0,\infty)}=t$, by \eqref{eq:claim} and \eqref{eq:formula} we see that \eqref{eq:gottschalk-hedlund} is violated for this particular $g$, and thus $\sigma_*$ does not act as the identity on the equivalence class of $\chi_{[\mathsf{1}]}$, which completes the proof.
\end{proof}

\begin{corollary}\label{cor:1-plus-sigma}
We have an isomorphism of abelian groups
\begin{equation}
(1+\sigma_*)\bigg(\frac{C(\underline{X},\bb{Z})}{\mathrm{im}(1-\varphi_*)}\bigg) \cong \bb{Z}[\tfrac{1}{2}].
\end{equation}
\end{corollary}

\begin{proof} 
By \Cref{thm:k_0-of-pf} together with \cite[Theorem~1.1]{Put89} we get $\frac{C(\underline{X},\bb{Z})}{\mathrm{im}(1-\varphi_*)} \cong \mathrm{K}_0(C(\underline{X})\rtimes_\varphi\bb{Z})\cong \bb{Z}[\tfrac{1}{2}]\oplus\bb{Z}$ with an isomorphism that carries $[1]$ to $(1,0)$. Under this isomorphism, $\sigma_*$ becomes an order $2$ automorphism on $\bb{Z}[\tfrac{1}{2}]\oplus\bb{Z}$ that maps $(1,0)$ to itself, and by $2$-divisibility of $\bb{Z}[\tfrac{1}{2}]$, it follows that $\sigma_*(q,0)=(q,0)$ for all $q\in\bb{Z}[\tfrac{1}{2}]$. Moreover, there are group homomorphisms $\alpha_1\colon\bb{Z}\to\bb{Z}[\tfrac{1}{2}]$ and $\alpha_2\colon\bb{Z}\to\bb{Z}$ such that 
\[
\sigma_*(0,n) = (\alpha_1(n),\alpha_2(n)),\quad n\in\bb{Z},
\]
and so
\begin{align*}
(0,n) &= \sigma_*^2(0,n) \\
&= \sigma_*(\alpha_1(n),\alpha_2(n)) \\
&= (\alpha_1(n+\alpha_2(n)), \alpha_2^2(n)), \quad n\in\bb{Z}.
\end{align*}
We then have that $\alpha_2^2=\mathrm{id}_\bb{Z}$, and so $\alpha_2=\mathrm{id}_\bb{Z}$ or $\alpha_2=-\mathrm{id}_\bb{Z}$. If however $\alpha_2=\mathrm{id}_\bb{Z}$, then $2\alpha_1(n)=\alpha_1(2n)=0$ for all $n\in\bb{Z}$ and so $\alpha_1=0$ (by $2$-divisibility of $\bb{Z}[\tfrac{1}{2}]$), whence $\sigma_*=\mathrm{id}$, which is impossible by the preceding lemma. We thus have $\alpha_2=-\mathrm{id}_\bb{Z}$, hence for any $(q,n)\in\bb{Z}[\tfrac{1}{2}]\oplus\bb{Z}$ we have
\begin{align*}
(1+\sigma_*)(q,n) &= (q,n) + \sigma_*(q,n ) \\
&= (q,n) + (q+\alpha_1(n), -n) \\
&= (2q+\alpha_1(n),0)
\end{align*}
which shows that $(1+\sigma_*)(\bb{Z}[\tfrac{1}{2}]\oplus\bb{Z})\subset\bb{Z}[\tfrac{1}{2}]\oplus 0$. The reverse inclusion is now also immediate since $(q,0) = (1+\sigma_*)(q/2,0)$, and thus we conclude that $(1+\sigma_*)(\bb{Z}[\tfrac{1}{2}]\oplus\bb{Z}) = \bb{Z}[\tfrac{1}{2}]\oplus0\cong\bb{Z}[\tfrac{1}{2}]$ as we wanted.
\end{proof}

We are now headed towards computing the $\mathrm{K}$-theory of the crossed product $\mathrm{C}^\ast$-algebra of the action $(\varphi,\sigma)\colon\bb{Z}\rtimes\bb{Z}_2\acts \underline{X}$ of \Cref{cor:free-min-dihedral-pf}. 

\begin{remark}\label{rmk:erratum}
The $\mathrm{K}$-theory of crossed products of Cantor minimal $\bb{Z}\rtimes\bb{Z}_2$ systems is treated in \cite[Section~4]{Tho10}, building on \cite{BraEvaKis93} (and mainly using \cite{Nat85}). In particular, let $(\varphi,\sigma)\colon\bb{Z}\rtimes\bb{Z}_2\acts \Omega$ be a free action on a Cantor space $\Omega$ with $\varphi\colon\bb{Z}\acts \Omega$ minimal and let $A\coloneqq C(\Omega)\rtimes(\bb{Z}\rtimes\bb{Z}_2)$. In \cite[Theorem~4.42]{Tho10} it is shown that  
\begin{equation}\label{eq:thomsen}
\mathrm{K}_0(A) \cong \bb{Z}_2\oplus (1+\sigma_*)(C(\Omega,\bb{Z})/\mathrm{im}(1-\varphi_*))
\end{equation}
as abelian groups. This adjusts the respective calculation from \cite{BraEvaKis93}, where the summand $\bb{Z}_2$ was overlooked in the case of absence of fixed points, cf.\ \cite[Remark~4.43]{Tho10}. In \cite[Equation~(4.24)]{Tho10} Thomsen tacitly makes an identification when describing $\mathrm{K}_0(A)$ as the cokernel of the map $(i_{1*},i_{2*})$ (as opposed to the cokernel of $(i_{1*}, -i_{2*})$, in accordance with the exact sequences in \cite[Theorem~A1]{Nat85} and \cite[Lemma~4.2]{BraEvaKis93}; cf.\ \cite{Kum90}). For our purposes it is important to keep track of where $[1_A]_0$ is mapped through the isomorphism in \eqref{eq:thomsen}, which is why below we revisit Thomsen's proof of \cite[Theorem~4.42]{Tho10}, yet bypassing the aforementioned identification.

Note first that freeness ensures that there are clopen sets $K, L\subset \Omega$ such that
\begin{equation}\label{eq:free-involutions}
K\sqcup \sigma(K) = \Omega \text{ and }L\sqcup\varphi\sigma(L)=\Omega,
\end{equation}
cf.\ \cite[Lemma~4.38]{Tho10}. Just as in \cite[p.~302]{Tho10}, set
\begin{align*}
G_\sigma &\coloneqq \{f\in C(\Omega,\bb{Z}): f\circ\sigma =f\},\\
G_{\varphi\sigma}&\coloneqq \{f\in C(\Omega,\bb{Z}): f\circ\varphi\sigma =f\}.
\end{align*}
We have isomorphisms $\mathrm{K}_0(C(\Omega)\rtimes_\sigma\bb{Z}_2)\cong G_\sigma$ and $\mathrm{K}_0(C(\Omega)\rtimes_{\varphi\sigma}\bb{Z}_2)\cong G_{\varphi\sigma}$ satisfying $[\chi_E]_0 \mapsto \chi_E+\chi_{\sigma(E)}$ and $[\chi_E]_0\mapsto \chi_E+\chi_{\varphi\sigma(E)}$ respectively, for any clopen set $E \subset\Omega$ (see \cite[p.~302]{Tho10}). By \cite[Lemma~4.2]{BraEvaKis93} (based on \cite{Nat85}) we thus have a commutative diagram with rows that are short exact sequences
\begin{equation*}
\begin{tikzcd}[column sep =2.7ex]
0\ar[r]& \mathrm{K}_0(C(\Omega))\ar[r,"\iota"]\ar[d,"\cong"] & \mathrm{K}_0(C(\Omega)\rtimes_\sigma\bb{Z}_2)\oplus \mathrm{K}_0(C(\Omega)\rtimes_{\varphi\sigma}\bb{Z}_2)\ar[d,"\cong"] \ar[r,"\pi"] & \mathrm{K}_0(A)\ar[d, "\mathrm{id}"] \ar[r]&0 \\
0\ar[r]& C(\Omega,\bb{Z}) \ar[r, "\bar{\iota}"] & G_\sigma\oplus G_{\varphi\sigma}\ar[r, "\bar{\pi}"]&\mathrm{K}_0(A)\ar[r]&0
\end{tikzcd}
\end{equation*}
where $\iota\coloneqq (i_{1*},-i_{2*})$, $\pi \coloneqq j_{1*}+j_{2*}$ with 
\[
C(\Omega)\xhookrightarrow{i_1}C(\Omega)\rtimes_\sigma\bb{Z}_2\xhookrightarrow{j_1}A \text{ and } C(\Omega)\xhookrightarrow{i_2}C(\Omega)\rtimes_{\varphi\sigma}\bb{Z}_2\xhookrightarrow{j_2}A
\]
being the canonical embeddings, 
\begin{equation}\label{eq:bariota}
\bar{\iota}(h) \coloneqq (h+h\circ\sigma, -h-h\circ\varphi\sigma),\quad h\in C(\Omega,\bb{Z}),
\end{equation}
and $\bar{\pi}$ satisfies
\begin{equation}\label{eq:barpi}
\bar{\pi}(\chi_E+\chi_{\sigma(E)},\chi_F+\chi_{\varphi\sigma(F)}) = [\chi_E]_0+ [ \chi_F]_0
\end{equation}
for all clopen sets $E,F\subset \Omega$.

Consider the map $\bb{Z}\ni z\mapsto (z,-z)\in G_\sigma\oplus G_{\varphi\sigma}$ (essentially as in \cite[p.~302]{Tho10}), which by minimality is seen to induce an injection $t\colon \bb{Z}_2\to \mathrm{K}_0(A)$ exactly as in the lines following \cite[Equation~(4.27)]{Tho10}. Note also that the $2$-torsion element $t(1)\in \mathrm{K}_0(A)$ is given by $t(1)=[\chi_K]_0 - [\chi_L]_0$ (cf.\ \eqref{eq:free-involutions}).

We consider the map $G_\sigma\oplus G_{\varphi\sigma}\ni(f,g)\mapsto f+g\in C(\Omega,\bb{Z})$ composed with the quotient map $C(\Omega,\bb{Z})\to C(\Omega,\bb{Z})/\mathrm{im}(1-\varphi_*)$, which vanishes on elements of the form $(h+h\circ\sigma, -h-h\circ\varphi\sigma)$, and thus induces a map $s\colon \mathrm{K}_0(A) \to C(\Omega,\bb{Z})/\mathrm{im}(1-\varphi_*)$. Upon making the necessary (albeit straightforward) modifications in the proof of \cite[Lemma~4.41]{Tho10} we see that
\begin{equation}\label{eq:thomsen-exact}
0\xlongrightarrow{} \bb{Z}_2\xlongrightarrow{t}\mathrm{K}_0(A)\xlongrightarrow{s}(1+\sigma_*)(C(\Omega,\bb{Z})/\mathrm{im}(1-\varphi_*))\xlongrightarrow{}0
\end{equation}
is short-exact. Finally, by exactness of \eqref{eq:thomsen-exact}, since the quotient is torsion free and since $\bb{Z}_2$ is cyclic, one applies Kulikov's theorem \cite[Theorem~24.5]{Fuc60} exactly as in the proof of \cite[Theorem~4.42]{Tho10} to obtain that \eqref{eq:thomsen-exact} is split exact, and in particular, as abelian groups,
\[
\mathrm{K}_0(A)\cong \bb{Z}_2\oplus (1+\sigma_*)(C(\Omega,\bb{Z})/\mathrm{im}(1-\varphi_*)).
\]
Note that $[1_A]_0 = \bar{\pi}(2,0)$ by \eqref{eq:barpi}, whence $s([1_A]_0)= (1+\sigma_*)([1])$ (which is equal to $2\cdot[1]$ since $\sigma_*([1])=[1]$), and so under the isomorphism above, $[1_A]_0$ is identified with either $(0,2\cdot[1])$ or $(1,2\cdot[1])$. We embrace this ambiguity for the time being also in the following proposition, which will give sufficiently precise information to prove our main result in the subsequent section.
\end{remark}

\begin{proposition}\label{thm:k-th-pf-dihedral}
Let $A\coloneqq C(\underline{X})\rtimes(\bb{Z}\rtimes\bb{Z}_2)$ be the crossed product associated to the action of \Cref{cor:free-min-dihedral-pf}. Then, $A$ is a classifiable $\mathrm{C}^\ast$-algebra (cf.\ \Cref{footnote}) that has unique trace, $\mathrm{K}_1(A)=0$, and $\big(\mathrm{K}_0(A),\mathrm{K}_0(A)_+,[1_A]_0\big)$ is isomorphic as a scaled ordered group either to
\begin{equation}
\big(\bb{Z}_2\oplus\bb{Z}[\tfrac{1}{2}], \{(x,q): x\in\bb{Z}_2, q>0\}\cup\{(0,0)\} , (0,1)\big),
\end{equation}
or to
\begin{equation}
\big(\bb{Z}_2\oplus\bb{Z}[\tfrac{1}{2}], \{(x,q): x\in\bb{Z}_2, q>0\}\cup\{(0,0)\} , (1,1)\big).
\end{equation}
Moreover, the pair $(C(\underline{X})\subset A)$ is a Cantor spectrum diagonal.
\end{proposition}

\begin{proof} 
Clearly $A$ is a unital and separable $\mathrm{C}^\ast$-algebra, which is also nuclear as the crossed product of an action of an amenable group on a compact metric space. Moreover $A$ is simple, as the crossed product of a free minimal action; see \cite{ArcSpi94}. By \cite{DowZha23}, since $\bb{Z}\rtimes\bb{Z}_2$ has polynomial growth, the underlying action is almost finite, and so $A$ is $\mc{Z}$-stable by \cite{Ker20}. (Alternatively, note that $A$ has finite nuclear dimension by \cite[Corollary~C]{BoeLi24}, and thus is $\mathcal{Z}$-stable by \cite{Win12}.) The pair $(C(\underline{X})\subset A)$ is a $\mathrm{C}^\ast$-diagonal since the action $\bb{Z}\rtimes\bb{Z}_2\acts \underline{X}$ is free, and thus $A$ also satisfies the UCT by \cite{BarLi17}. 

Note that the action $\varphi\colon \bb{Z}\acts \underline{X}$ is uniquely ergodic (this follows from the fact that $\varphi\colon \bb{Z}\acts X_\varrho$ is uniquely ergodic and that the two Cantor minimal systems are (strongly) orbit equivalent in the sense of \cite{GioPutSka95}, due to \Cref{thm:k_0-of-pf} and \Cref{prop:K-theory-rho}, and thus have affinely homeomorphic simplices of invariant measures \cite[Theorem~2.2]{GioPutSka95}). Since the invariant measures of the dynamical system $\bb{Z}\rtimes\bb{Z}_2\acts \underline{X}$ are a subset of the invariant measures of the action of the $\bb{Z}$-subgroup, and since any action of an amenable group admits an invariant measure, $\bb{Z}\rtimes\bb{Z}_2\acts \underline{X}$ is uniquely ergodic, and thus $A$ has a unique trace, since in this setting there is an affine homeomorphism between the space of tracial states on $A$ and the Borel probability $\bb{Z}\rtimes\bb{Z}_2$-invariant measures on $\underline{X}$ (cf.\ \cite[\S 4.1]{LiRen19}).

It follows from \cite[Theorem~4.42]{Tho10} that $\mathrm{K}_1(A)=0$, so it remains to compute the $\mathrm{K}_0$-group. Combining \Cref{cor:1-plus-sigma} with \Cref{rmk:erratum}, we have that $\mathrm{K}_0(A)\cong \bb{Z}_2\oplus \bb{Z}[\tfrac{1}{2}]$, and $[1_A]_0$ corresponds either to the element $(0,4)$ or to $(1,4)$. Moreover, since the map $\bb{Z}_2\oplus\bb{Z}[\tfrac{1}{2}]\ni(x,y)\mapsto(x,\tfrac{1}{4}y)\in\bb{Z}_2\oplus\bb{Z}[\tfrac{1}{2}]$ is a group automorphism that maps $(0,4)$ to $(0,1)$ and $(1,4)$ to $(1,1)$, we obtain a group isomorphism $\mathrm{K}_0(A)\cong\bb{Z}_2\oplus\bb{Z}[\tfrac{1}{2}]$ with $[1_A]_0$ identified either with $(0,1)$ or with $(1,1)$.

Since $A$ is simple, the ordered group $(\mathrm{K}_0(A),\mathrm{K}_0(A)_+)$ is simple, and by \cite[Corollary~4.33]{Tho10} $A$ has tracial topological rank zero (in the sense of Lin \cite[Definition~3.3.4]{RorSto02}), whence by \cite{Lin01} we have that the ordered group $(\mathrm{K}_0(A),\mathrm{K}_0(A)_+)$ is weakly unperforated. We can thus apply \cite[Theorem~6.8.5]{Bla98} to conclude that $\mathrm{K}_0(A)_+$ is equal to the set
\begin{equation}\label{eq:cone-determined}
\{g\in\mathrm{K}_0(A): \gamma(g)>0\text{ for all }\gamma\in\mathrm{S}(\mathrm{K}_0(A),\mathrm{K}_0(A)_+,[1_A]_0)\}\cup\{0\},
\end{equation}
where $\mathrm{S}(\mathrm{K}_0(A),\mathrm{K}_0(A)_+,[1_A]_0)$ denotes the set of states of the scaled ordered group $(\mathrm{K}_0(A),\mathrm{K}_0(A)_+,[1_A]_0)$. If $\gamma$ is a state on $\mathrm{K}_0(A)$, then it induces a group homomorphism $\bb{Z}_2\oplus\bb{Z}[\tfrac{1}{2}]\to\bb{R}$ upon composing with the isomorphism $\mathrm{K}_0(A)\cong\bb{Z}_2\oplus\bb{Z}[\tfrac{1}{2}]$. This map is necessarily $0$ on the $\bb{Z}_2$-summand since $\bb{R}$ is torsion free, and maps either $(1,1)$ or $(0,1)$ to $1$ (depending on which element corresponds to $[1_A]_0$): in either case, since the map vanishes on the $\bb{Z}_2$-summand, $(0,1)$ is mapped to $1$ and so this map is necessarily the canonical embedding of the $\bb{Z}[\tfrac{1}{2}]$-summand into $\bb{R}$. By this observation and \eqref{eq:cone-determined} we conclude that $\mathrm{K}_0(A)_+$ is carried onto $\{(x,q)\in\bb{Z}_2\oplus\bb{Z}[\tfrac{1}{2}]: q>0\}\cup\{(0,0)\}$ and the proof is complete.
\end{proof}

\begin{remark}
In order to expand the scope of our method to more examples like, say, other UHF algebras, it would be interesting to employ groupoid cohomology along the lines of \cite{BoeDeAGabWil23} in the framework of Matui's HK-conjecture; cf.\ \cite{Mat12}.
\end{remark}

\section{Non-standard Cantor spectrum diagonals in the CAR algebra}
\noindent We are now ready for the proof of \Cref{intro:thm1}. We denote by $(\mathrm{D}_{2^\infty} \subset \mathrm{M}_{2^\infty})$ the standard $\mathrm{C}^\ast$-diagonal of the CAR algebra obtained as the inductive limit of diagonal $2^n\times 2^n$ matrices.

\begin{theorem}\label{thm:main1}
The CAR algebra $\mathrm{M}_{2^\infty}$ admits a Cantor spectrum diagonal $(D \subset \mathrm{M}_{2^\infty})$ that is not an AF diagonal.
\end{theorem}

\begin{proof}
Let $(\varphi,\sigma)\colon\bb{Z}\rtimes\bb{Z}_2\acts \underline{X}$ be the action of \Cref{cor:free-min-dihedral-pf} and set $A\coloneqq C(\underline{X})\rtimes_{(\varphi,\sigma)}(\bb{Z}\rtimes\bb{Z}_2)$. Set $B\coloneqq A\otimes \mathrm{M}_{2^\infty}$ and $D\coloneqq C(\underline{X})\otimes \mathrm{D}_{2^\infty}$. The spectrum of $D$ is the Cartesian product of the spectra of its factors, whence $D$ has Cantor spectrum. Moreover, $(D \subset B)$ is a $\mathrm{C}^\ast$-diagonal \cite[Lemma~5.1]{BarLi17}.

We first show that $B\cong \mathrm{M}_{2^\infty}$ which we do via classification theory. Note that $B$ is unital, simple and separable, and it is moreover nuclear as the tensor product of two nuclear $\mathrm{C}^\ast$-algebras. Since $B$ admits a Cartan subalgebra (namely $D$) it satisfies the UCT \cite{BarLi17}. By its definition $B$ is $\mathrm{M}_{2^\infty}$-stable, which implies that $B$ is $\mc{Z}$-stable, and since both $A$ and $\mathrm{M}_{2^\infty}$ have unique trace, so does $B$. As abelian groups, the $\mathrm{K}$-theory of $B$ can be computed via the K{\"u}nneth formula \cite[Theorem~23.1.3]{Bla98}: since the $\mathrm{K}$-theory of $\mathrm{M}_{2^\infty}$ is torsion free and since $\mathrm{K}_1(A)=\mathrm{K}_1(\mathrm{M}_{2^\infty})=0$ (by \Cref{thm:k-th-pf-dihedral}), we have
\[
\mathrm{K}_1(B) \cong \big(\mathrm{K}_0(A)\otimes_{\bb{Z}}\mathrm{K}_1(\mathrm{M}_{2^\infty})\big)\oplus\big(\mathrm{K}_1(A)\otimes_{\bb{Z}}\mathrm{K}_0(\mathrm{M}_{2^\infty})\big) = 0.
\]
Also, using at the fourth line below that the tensor product of abelian groups is distributive with respect to direct sums and the fact that $\bb{Z}_2$ has $2$-torsion while $\bb{Z}[\tfrac{1}{2}]$ is $2$-divisible, we have group isomorphisms
\begin{align}
\mathrm{K}_0(B) &\cong \big(\mathrm{K}_0(A) \otimes_\bb{Z}\mathrm{K}_0(\mathrm{M}_{2^\infty})\big)\oplus\big(\mathrm{K}_1(A)\otimes_{\bb{Z}}\mathrm{K}_1(\mathrm{M}_{2^\infty})\big) \label{eq:l1} \\
&\cong \mathrm{K}_0(A) \otimes_\bb{Z}\mathrm{K}_0(\mathrm{M}_{2^\infty}) \label{eq:l2}\\
&\cong \big(\bb{Z}_2\oplus\bb{Z}[\tfrac{1}{2}]\big)\otimes_{\bb{Z}}\bb{Z}[\tfrac{1}{2}]\label{eq:l3}\\
&\cong \bb{Z}[\tfrac{1}{2}]\otimes_{\bb{Z}}\bb{Z}[\tfrac{1}{2}] \label{eq:l4}\\
&\cong \bb{Z}[\tfrac{1}{2}],\label{eq:l5}
\end{align}
with \eqref{eq:l5} the isomorphism satisfying $q_1\otimes q_2\mapsto q_1q_2$ on elementary tensors. Note that $[1_B]_0$ is mapped to the elementary tensor $[1_A]_0\otimes[1_{\mathrm{M}_{2^\infty}}]_0$ via the isomorphisms in \eqref{eq:l1} and \eqref{eq:l2} (cf.\ \cite[\S 23.1]{Bla98} and \cite{Sch82}) and by \Cref{thm:k-th-pf-dihedral} this is in turn mapped to either $(0,1)\otimes 1$ or $(1,1)\otimes 1$ via the isomorphism in \eqref{eq:l3}. In either case this is mapped to $1\otimes1$ via the isomorphism in \eqref{eq:l4}, which is then mapped to $1$ via the isomorphism in \eqref{eq:l5}.

As for the positive cone, we argue like in the proof of \Cref{thm:k-th-pf-dihedral}. Since $B$ is a simple (and stably finite) $\mathrm{C}^\ast$-algebra, $(\mathrm{K}_0(B), \mathrm{K}_0(B)_+)$ is a simple ordered group, and as $B$ is the tensor product of a tracially AF algebra (in the sense of Lin \cite{Lin01}) with a UHF algebra, $B$ is also tracially AF by \cite[Proposition~5.9]{Lin01}, whence $(\mathrm{K}_0(B),\mathrm{K}_0(B)_+)$ is weakly unperforated. By \cite[Theorem~6.8.5]{Bla98}, we have that $\mathrm{K}_0(B)_+$ is equal to the set
\[
\{g\in \mathrm{K}_0(B): \gamma(g)>0\text{ for all }\gamma\in\mathrm{S}(\mathrm{K}_0(B),\mathrm{K}_0(B)_+,[1_B]_0)\}\cup\{0\}.
\]
Any state on $\mathrm{K}_0(B)$ induces a group homomorphism $\gamma\colon\bb{Z}[\tfrac{1}{2}]\to \bb{R}$ by composing with the isomorphisms of \eqref{eq:l1}--\eqref{eq:l5} and this will satisfy $\gamma(1)=1$ (since $[1_B]_0$ is carried to $1$). There is however a unique such homomorphism by $2$-divisibility, namely the canonical embedding of $\bb{Z}[\tfrac{1}{2}]$ into $\bb{R}$. We conclude that $\mathrm{K}_0(B)_+$ is carried to $\bb{Z}[\tfrac{1}{2}]_{\ge0}$ by the isomorphisms in \eqref{eq:l1}--\eqref{eq:l5}, and thus the scaled ordered group $(\mathrm{K}_0(B),\mathrm{K}_0(B)_+,[1_B]_0)$ is isomorphic to $(\bb{Z}[\tfrac{1}{2}],\bb{Z}[\tfrac{1}{2}]_{\ge0},1)$. By classification (see \cite{Win18,Whi23}), we now have that $B\cong \mathrm{M}_{2^\infty}$.

To show that $(D \subset B)$ is not an AF diagonal, it suffices to show that $(D \subset B)$ is not conjugate to $(\mathrm{D}_{2^\infty}\subset \mathrm{M}_{2^\infty})$ (since all AF diagonals in $\mathrm{M}_{2^\infty}$ are conjugate \cite[Theorem~5.7]{Pow92}; also cf.\ \cite[Remark~4.2]{LiLiaWin23}). To that end, note that we have the intermediate sub-$\mathrm{C}^\ast$-algebra
\[
D \subset (C(\underline{X})\rtimes_\varphi\bb{Z})\otimes \mathrm{M}_{2^\infty} \subset B
\]
and note also that $\mathrm{K}_1((C(\underline{X})\rtimes_\varphi\bb{Z})\otimes \mathrm{M}_{2^\infty})\cong\bb{Z}$ (cf.\ \cite{GioPutSka95}). In particular, $(C(\underline{X})\rtimes_\varphi\bb{Z})\otimes \mathrm{M}_{2^\infty}$ is not an AF algebra. However, if $C$ is any intermediate sub-$\mathrm{C}^\ast$-algebra $(\mathrm{D}_{2^\infty} \subset C \subset \mathrm{M}_{2^\infty})$, then $C$ is necessarily AF by the main result of \cite{ArcKum86}. This completes the proof.
\end{proof}

\begin{remark}\label{rmk:new}
In order to confirm that the diagonal above is not AF, instead of using \cite{ArcKum86} one could also observe that there is a unitary normaliser (coming from the paper-folding subshift), such that the induced homeomorphism has infinite orbits, which is not possible for an AF diagonal by a compactness argument.

\end{remark}

Using the theory of \emph{diagonal dimension} developed by Li, Liao and the second named author in \cite{LiLiaWin23}, we can use \Cref{thm:main1} to obtain countably many pairwise non-conjugate $\mathrm{C}^\ast$-diagonals with Cantor spectra in the CAR algebra.

\begin{remark}
By \cite[Theorem~5.4]{LiLiaWin23}, for $G\acts X$ a free action of a countable group on a Cantor space, $\dim_\mathrm{diag}(C(X)\subset C(X)\rtimes_\mathrm{r}G)$ is equal to the \emph{tower dimension} of $G\acts X$, as defined by Kerr in \cite[Definition~4.3]{Ker20}, which further agrees with the \emph{dynamic asymptotic dimension} of $G\acts X$, denoted by $\mathrm{dad}(G\acts X)$, due to \cite[Theorem~5.14]{Ker20}, since $X$ is a Cantor space. It is clear from the definition of the dynamic asymptotic dimension (see \cite[Definition~2.1]{GueWilYu17}; cf.\ \cite[Definition~5.3]{Ker20}) that if $H$ is a subgroup of $G$, then $\mathrm{dad}(H\acts X)\le\mathrm{dad}(G\acts X)$ (see also \cite[Lemma~2.2]{Boe24}). In particular,

\begin{equation}\label{eq:shortcut}
\dim_\mathrm{diag}(C(X)\subset C(X)\rtimes_\mathrm{r}H) \le \dim_\mathrm{diag}(C(X) \subset C(X)\rtimes_\mathrm{r}G).
\end{equation}

\end{remark}

For the course of the following proof we write $A^{\otimes n}$ for the $n$-fold (minimal) tensor product of a unital $\mathrm{C}^\ast$-algebra $A$ with itself, and we use $A^{\otimes \infty}$ for the infinite tensor product of $A$ with itself, defined as the inductive limit of the system $A\to A^{\otimes 2} \to A^{\otimes 3}\to\dots$ with the connecting maps given by $x\mapsto x\otimes 1_A$. For a compact space $X$ we write $X^n$ for the $n$-fold Cartesian product of $X$ with itself and we set $X^\infty\coloneqq\prod_\bb{N}X$. For $\mathrm{C}^\ast$-pairs, we write $(D_1\subset A_1) \cong (D_2\subset A_2)$ when there is a $^*$-isomorphism from $A_1$ to $A_2$ carrying $D_1$ onto $D_2$.

\begin{theorem}\label{thm:countably-many}
For each $n\in \{0,1,2,\dots,\infty\}$, there is a Cantor spectrum diagonal $(D^{(n)} \subset \mathrm{M}_{2^\infty})$ such that
\[
\dim_{\mathrm{diag}}(D^{(n)} \subset \mathrm{M}_{2^\infty})=n.
\]

In particular, the CAR algebra admits countably many pairwise non-conjugate $\mathrm{C}^\ast$-diagonals each of which has Cantor spectrum.
\end{theorem}

\begin{proof}
The case $n=0$ is covered by the canonical AF diagonal, since $\dim_{\mathrm{diag}}(\mathrm{D}_{2^\infty}\subset\mathrm{M}_{2^\infty})=0$ (see \cite[Theorem~4.1 and Remark~4.2]{LiLiaWin23}).

By (the proof of) \Cref{thm:main1} we have a free action $\bb{Z}\rtimes\bb{Z}_2\acts X$ of the infinite dihedral group on the Cantor space, such that $D\coloneqq C(X)\otimes \mathrm{D}_{2^\infty}$ is a $\mathrm{C}^\ast$-diagonal in $A\coloneqq (C(X)\rtimes(\bb{Z}\rtimes\bb{Z}_2))\otimes \mathrm{M}_{2^\infty}\cong \mathrm{M}_{2^\infty}$.

For $1\le n<\infty$ let $D^{(n)}\coloneqq D^{\otimes n} \subset A^{\otimes n}\cong \mathrm{M}_{2^\infty}^{\otimes n}\cong \mathrm{M}_{2^\infty}$ and note that each pair $(D^{(n)} \subset \mathrm{M}_{2^\infty})$ is a Cantor spectrum diagonal \cite[Lemma~5.1]{BarLi17}. Upon re-arranging the tensor factors, and since $(\mathrm{D}_{2^\infty}^{\otimes n} \subset \mathrm{M}_{2^\infty}^{\otimes n})\cong (\mathrm{D}_{2^\infty}\subset \mathrm{M}_{2^\infty})$,
\[
(D^{(n)}\subset \mathrm{M}_{2^\infty}) \cong (C(X^n)\otimes \mathrm{D}_{2^\infty } \subset (C(X^n)\rtimes(\bb{Z}\rtimes\bb{Z}_2)^n)\otimes \mathrm{M}_{2^\infty})
\]
and so by \cite[Theorem~3.1~(ii)]{LiLiaWin23} 
\begin{align*}
\dim_{\mathrm{diag}}(D^{(n)}\subset \mathrm{M}_{2^\infty}) \le \dim_{\mathrm{diag}}(C(X^n) \subset C(X^n)\rtimes(\bb{Z}\rtimes\bb{Z}_2)^n).
\end{align*}
Now by \cite[Proposition~7.2]{LiLiaWin23}, and since $(\bb{Z}\rtimes\bb{Z}_2)^n$ is finitely generated, virtually abelian with asymptotic dimension $n$,\footnote{Note that $\mathrm{asdim}((\bb{Z}\rtimes\bb{Z}_2)^n)\le n\cdot\mathrm{asdim}(\bb{Z}\rtimes\bb{Z}_2)=n$, where the last equality follows from the fact that $\bb{Z}\rtimes\bb{Z}_2$ is virtually $\bb{Z}$ and $\mathrm{asdim}(\bb{Z})=1$. Also, since $\bb{Z}^n$ is a subgroup of $(\bb{Z}\rtimes\bb{Z}_2)^n$, we have that $n=\mathrm{asdim}(\bb{Z}^n)\le\mathrm{asdim}((\bb{Z}\rtimes\bb{Z}_2)^n)$.} 
\begin{equation}\label{eq:another}
\mathrm{dim}_{\mathrm{diag}}(C(X^n)\subset C(X^n)\rtimes(\bb{Z}\rtimes\bb{Z}_2)^n)=n
\end{equation}
and thus 
\begin{equation*}\tag{4.7i}
\dim_\mathrm{diag}(D^{(n)}\subset \mathrm{M}_{2^\infty})\le n.
\end{equation*}

Let $\Gamma$ be the locally finite group $\bigoplus_\bb{N}\bb{Z}_2$. There is a free minimal action $\Gamma\acts Y$, with $Y$ a Cantor space, such that $(C(Y)\subset C(Y)\rtimes\Gamma)\cong (\mathrm{D}_{2^\infty}\subset \mathrm{M}_{2^\infty})$, see \cite[II.10.4.12~(iii)]{Bla06}. Again upon re-arranging the tensor factors and since $(\mathrm{D}_{2^\infty}^{\otimes n} \subset \mathrm{M}_{2^\infty}^{\otimes n})\cong (\mathrm{D}_{2^\infty}\subset \mathrm{M}_{2^\infty})$,
\begin{equation*}\tag{4.7a}
(D^{(n)} \subset \mathrm{M}_{2^\infty}) \cong (C(X^n\times Y)\subset C(X^n\times Y)\rtimes((\bb{Z}\rtimes\bb{Z}_2)^n\times \Gamma)).
\end{equation*}
Applying \eqref{eq:shortcut} with $H= (\bb{Z}\rtimes\bb{Z}_2)^n\times\{0_\Gamma\} \le (\bb{Z}\rtimes\bb{Z}_2)^n\times \Gamma = G$ therein, together with (4.7a),
\begin{align*}\tag{4.7b}
&\dim_\mathrm{diag}(C(X^n)\otimes \mathrm{D}_{2^\infty} \subset (C(X^n)\rtimes(\bb{Z}\rtimes\bb{Z}_2)^n)\otimes\mathrm{D}_{2^\infty})\\ &\le\dim_{\mathrm{diag}}(D^{(n)}\subset\mathrm{M}_{2^\infty}).
\end{align*}
Set $B\coloneqq (C(X^n)\rtimes(\bb{Z}\rtimes\bb{Z}_2)^n)\otimes\mathrm{D}_{2^\infty}$ and let $\vartheta$ be a character on $\mathrm{D}_{2^\infty}$. The map $\mathrm{id}\otimes\vartheta\colon B \to C(X^n)\rtimes(\bb{Z}\rtimes\bb{Z}_2)^n$ is a surjective $^*$-homomorphism that carries $C(X^n)\otimes \mathrm{D}_{2^\infty}$ onto $C(X^n)$, and so by \cite[Theorem~3.2~(v)]{LiLiaWin23} we have
\begin{equation}\label{eq:some}
\dim_{\mathrm{diag}}(C(X^n) \subset C(X^n)\rtimes(\bb{Z}\rtimes\bb{Z}_2)^n)\le \dim_{\mathrm{diag}}(C(X^n)\otimes \mathrm{D}_{2^\infty} \subset B),
\end{equation}
whence by combining \eqref{eq:another} with \eqref{eq:some}, we have 
\begin{equation*}\tag{4.7ii}
n \le \dim_\mathrm{diag}(D^{(n)} \subset \mathrm{M}_{2^\infty}).
\end{equation*}
Putting (4.7i) and (4.7ii) together, we see that $\dim_{\mathrm{diag}}(D^{(n)}\subset \mathrm{M}_{2^\infty})=n$ as wanted.

Lastly, set $D^{(\infty)}\coloneqq D^{\otimes \infty} \subset  A^{\otimes \infty} \cong \mathrm{M}_{2^\infty}^{\otimes \infty} \cong \mathrm{M}_{2^\infty}$ which is a $\mathrm{C}^\ast$-diagonal arising as $(C(X^\infty\times Y) \subset C(X^\infty\times Y)\rtimes (\bigoplus_\bb{N}(\bb{Z}\rtimes\bb{Z}_2)\times \Gamma))$ for the product of the coordinate-wise action $\bigoplus_\bb{N}(\bb{Z}\rtimes\bb{Z}_2)\acts X^\infty$ with $\Gamma\acts Y$ (cf.\ \cite[Lemma~5.2]{BarLi17}). For $n\in\bb{N}$
\[
(D^{(\infty)} \subset \mathrm{M}_{2^\infty})\cong (D^{\otimes n} \otimes D^{\otimes \infty} \subset A^{\otimes n} \otimes A^{\otimes \infty})
\]
and the latter pair can be written as a canonical crossed product pair of an action of $(\bb{Z}\rtimes\bb{Z}_2)^n\times\Gamma\times\big(\bigoplus_\bb{N}(\bb{Z}\rtimes\bb{Z}_2)\times\Gamma\big)$. Applying \eqref{eq:shortcut} for the subgroup $(\bb{Z}\rtimes\bb{Z}_2)^n$ and arguing as in the preceding paragraph, we obtain $n\le \dim_\mathrm{diag}(D^{(\infty)} \subset \mathrm{M}_{2^\infty})$. Since $n\in\bb{N}$ was arbitrary, this proves that $\dim_\mathrm{diag}(D^{(\infty)}\subset \mathrm{M}_{2^\infty})=\infty$ as desired.
\end{proof}

\begin{remark}\label{rmk:dimdiagp}
The theorem above in particular says that an AF algebra may contain a diagonal which itself is AF and such that its inclusion has nonzero diagonal dimension, thus resolving a problem that was raised in \cite[Remark~6.10]{LiLiaWin23}.
\end{remark}

\begin{remark}\label{rmk:graphs}
In \cite[Section~6]{EvaSim12}, the authors construct a $2$-graph $\Lambda_{\mathrm{II}}$ (see \cite[Section~2]{EvaSim12} for the relevant definitions) such that  $\mathrm{C}^\ast(\Lambda_\mathrm{II})$ is unital, separable, simple, nuclear, in the UCT class, has a unique trace and contains a projection $p$ such that the corner $p\mathrm{C}^\ast(\Lambda_\mathrm{II})p$ has the same ordered $\mathrm{K}$-theory as the CAR algebra, and which contains a Cartan sub-$\mathrm{C}^\ast$-algebra $D_0\subset p\mathrm{C}^\ast(\Lambda_{\mathrm{II}})p$ with Cantor spectrum that is not a $\mathrm{C}^\ast$-diagonal (see the discussion after \cite[Remark~6.13]{EvaSim12}). While the conditions on $\mathrm{C}^\ast(\Lambda_\mathrm{II})$ are not enough to apply the classification theorem and deduce that $p\mathrm{C}^\ast(\Lambda_{\mathrm{II}})p$ is in fact isomorphic to $\mathrm{M}_{2^\infty}$, this obstacle is surpassed by considering $p\mathrm{C}^\ast(\Lambda_{\mathrm{II}})p\otimes \mathrm{M}_{2^\infty}$, which is $\mc{Z}$-stable due to the tensor factor $\mathrm{M}_{2^\infty}$, and has the same $\mathrm{K}$-theory (this follows from the K{\"u}nneth formula)  and traces as $\mathrm{M}_{2^\infty}$. This leads to the somewhat surprising conclusion that $\mathrm{M}_{2^\infty}$ contains Cartan sub-$\mathrm{C}^\ast$-algebras with Cantor spectra that are not $\mathrm{C}^\ast$-diagonals, which are obtained as the tensor products of $D_0$ either with $\mathrm{D}_{2^\infty}$ or with any of the diagonals obtained in \Cref{thm:countably-many}.
\end{remark}

\renewcommand{\thetheorem}{\Alph{theorem}}
\setcounter{theorem}{0}

\end{document}